\documentclass{amsart}

\usepackage{amsmath}
\usepackage{amsthm}
\usepackage{amssymb}
\usepackage{mathrsfs}
\usepackage[all]{xy}

\title{Homological perturbation theory for algebras over operads}
\author{Alexander Berglund}
\address{Department of Mathematics,
         Stockholm University,
         SE-106 91,
         Sweden}
\email{alexb@math.su.se}

\thanks{Supported by the Danish National Research Foundation through the Centre for Symmetry and Deformation (DNRF92).}

\newtheorem{theorem}{Theorem}[section]
\newtheorem{proposition}[theorem]{Proposition}

\newtheorem{corollary}[theorem]{Corollary}

\theoremstyle{definition}
\newtheorem{definition}[theorem]{Definition}
\newtheorem{remark}[theorem]{Remark}

\newcommand{\antishriek}{\text{!`}}
\newcommand{\kk}{\Bbbk}
\newcommand{\SDR}[5]{\xymatrix{*[r]{#1} \ar@<1ex>[r]^-{#3} \ar@(ul,dl)[]_{#5} & #2 \ar@<1ex>[l]^-{#4}}}
\newcommand{\bigSDR}[5]{\xymatrix{*[r]{#1} \ar@<1ex>[rr]^-{#3} \ar@(ul,dl)[]_{#5} && #2 \ar@<1ex>[ll]^-{#4}}}
\newcommand{\Disc}[2]{D^{#2}(#1)}
\newcommand{\dash}{-}
\newcommand{\Cat}{\mathscr{C}}
\newcommand{\Acat}{\mathscr{A}}
\newcommand{\OpP}{\mathcal{P}}
\newcommand{\OpS}{\mathcal{S}}
\newcommand{\Coop}{\mathcal{C}}
\newcommand{\Diag}{\mathscr{D}}
\newcommand{\Ind}{\operatorname{Ind}}
\newcommand{\Talg}{T}
\newcommand{\Ob}{\operatorname{Ob}}
\newcommand{\Symgp}[1]{{\Sigma_{#1}}}
\newcommand{\Symgpd}{\Sigma}
\newcommand{\Op}{\mathcal{O}}

\newcommand{\defn}[1]{\emph{#1}}
\newcommand{\set}[2]{\left\{ #1 \mid #2 \right\}}
\newcommand{\numbers}{\mathbb{N}}
\newcommand{\rationals}{\mathbb{Q}}
\newcommand{\tensor}{\otimes}
\newcommand{\Hom}{\operatorname{Hom}}
\renewcommand{\H}{\operatorname{H}}
\newcommand{\Th}[1]{T_\numbers(#1)}
\newcommand{\Ths}[1]{T_\Symgpd(#1)}
\newcommand{\thzr}{\mathbf{0}}
\newcommand{\tho}{\mathbf{1}}
\newcommand{\thf}{\mathbf{f}}
\newcommand{\thg}{\mathbf{g}}
\newcommand{\thh}{\mathbf{h}}
\newcommand{\tht}{\mathbf{t}}
\newcommand{\thd}{\mathbf{d}}
\newcommand{\thr}{\mathbf{r}}
\newcommand{\thl}{\mathbf{l}}
\newcommand{\thq}{\mathbf{q}}
\newcommand{\thF}{\mathbf{F}}
\newcommand{\thG}{\mathbf{G}}
\newcommand{\thH}{\mathbf{H}}
\newcommand{\thFp}{\mathbf{F'}}
\newcommand{\thGp}{\mathbf{G'}}
\newcommand{\thHp}{\mathbf{H'}}
\newcommand{\thS}{\mathbf{\Sigma}}
\newcommand{\thDiag}{\underline{\Diag}}
\newcommand{\thx}{\mathbf{x}}
\newcommand{\thy}{\mathbf{y}}
\newcommand{\thz}{\mathbf{z}}
\newcommand{\thfp}{\mathbf{f'}}
\newcommand{\thgp}{\mathbf{g'}}
\newcommand{\thhp}{\mathbf{h'}}
\newcommand{\thtp}{\mathbf{t'}}
\newcommand{\sh}{\thh^\Sigma}
\newcommand{\shn}[1]{\thh_{#1}^{\Sigma}}
\newcommand{\hh}{\thh^{der}}
\newcommand{\Qb}{\mathbf{q}}
\newcommand{\hhn}[1]{\thh_n^{der}}
\newcommand{\Qn}[1]{\thq_n}
\newcommand{\coeff}[2]{Q_{#2}^{#1}}
\newcommand{\As}{\mathtt{A}}
\newcommand{\Asa}{\As^\antishriek}

\newcommand{\BB}{\mathbb{B}}

\newcommand{\HH}[1]{\thh_{#1}}
\newcommand{\OpI}{\mathcal{I}}

\newcommand{\End}[1]{\mathcal{E}nd_{#1}}

\newcommand{\ttree}{\xygraph{
!{<0pt,0pt>;<-11pt,0pt>:<0pt,11pt>::}
!{(-2,2)}*{{}_g}="a"
!{(2,2)}*{{}_g}="b"
!{(0,0)}*+{{}_{m_2}}="c"
!{(0,-2)}*+{{}_h}="d"
"a"-"c"
"b"-"c"
"c"-"d"
}}

\newcommand{\rtree}{\xygraph{
!{<0pt,0pt>;<11pt,0pt>:<0pt,11pt>::}
!{(-4,3)}*{{}_g}="a"
!{(0,3)}*{{}_g}="b"
!{(4,3)}*{{}_g}="c"
!{(-2,1)}*+{{}_{m_2}}="d"
!{(0,-1)}*+{{}_{m_2}}="e"
!{(0,-3)}*+{{}_h}="f"
"a"-"d"
"b"-"d"
"c"-"e"
"d"-"e"|{h}
"e"-"f"
}}

\newcommand{\ltree}{\xygraph{
!{<0pt,0pt>;<-11pt,0pt>:<0pt,11pt>::}
!{(-4,3)}*{{}_g}="a"
!{(0,3)}*{{}_g}="b"
!{(4,3)}*{{}_g}="c"
!{(-2,1)}*+{{}_{m_2}}="d"
!{(0,-1)}*+{{}_{m_2}}="e"
!{(0,-3)}*+{{}_h}="f"
"a"-"d"
"b"-"d"
"c"-"e"
"d"-"e"|{h}
"e"-"f"
}}

\newcommand{\mtree}{\xygraph{
!{<0pt,0pt>;<11pt,0pt>:<0pt,11pt>::}
!{(-4,3)}*{{}_g}="a"
!{(0,3)}*{{}_g}="b"
!{(4,3)}*{{}_g}="c"
!{(0,-0.5)}*+{{}_{m_3}}="d"
!{(0,-3)}*+{{}_h}="e"
"a"-"d"
"b"-"d"
"c"-"d"
"d"-"e"
}}

\begin{document}

\begin{abstract}
We extend homological perturbation theory to encompass algebraic structures governed by operads and cooperads. The main difficulty is to find a suitable notion of algebra homotopy that generalizes to algebras over operads $\Op$. To solve this problem, we introduce \emph{thick maps of $\Op$-algebras} and special thick maps that we call \emph{pseudo-derivations} that serve as appropriate generalizations of algebra homotopies for the purposes of homological perturbation theory.

As an application, we derive explicit formulas for transferring $\Omega(\Coop)$-algebra structures along contractions, where $\Coop$ is any connected cooperad in chain complexes. This specializes to transfer formulas for $\Op_\infty$-algebras for any Koszul operad $\Op$, in particular for $A_\infty$, $C_\infty$, $L_\infty$ and $G_\infty$-algebras. A key feature is that our formulas are expressed in terms of the compact description of $\Omega(\Coop)$-algebras as coderivation differentials on cofree $\Coop$-coalgebras. Moreover, we get formulas not only for the transferred structure and a structure on the inclusion, but also for structures on the projection and the homotopy.
\end{abstract}

\maketitle

\section{Introduction}
Perturbation methods have proved to be very useful in algebraic topology, homological algebra, and deformation theory \cite{E.Brown,Brown,Gugenheim,Halperin-Stasheff,Hess,HK,HS,JLS,LamS,Schlessinger-Stasheff}. It has been pointed out that homological perturbation theory suffers from the defect of not handling well algebra structures where symmetries play a role, such as Lie or commutative algebras, or more generally algebras over an operad $\Op$ in chain complexes. An obstruction has been the lack of a good notion of `algebra homotopy' for these types of algebras, see for instance \cite[Remark, end of \S 2.2]{GLS2}, \cite{HS}, or \cite[Remark 12.2]{Huebschmann-A}.

The goal of this paper is to solve this problem. We do this by introducing the notion of \emph{thick maps} of $\Op$-algebras. Thick maps are a simultaneous generalization of morphisms and derivations. We single out special thick maps that we call \emph{pseudo-derivations}, and we show that these are appropriate generalizations of algebra homotopies for the purposes of homological perturbation theory. Our main technical results are the `$\Op$-algebra perturbation lemma' (Theorem \ref{thm:PL}) and the `$\Op$-algebra tensor trick' (Theorem \ref{thm:tt}).

A classical application of homological perturbation theory is the streamlined proof of the transfer theorem for $A_\infty$-algebras, see \cite[\S 4.2]{GLS2}. Due to the defect mentioned above, it has not been possible to treat more general types of strong homotopy algebras in the same way. But as an application of the results presented here, we obtain simple and explicit formulas for transferring $\Op_\infty$-algebra structures along contractions, or more generally $\Omega(\Coop)$-algebra structures for any connected cooperad $\Coop$, see Theorem \ref{thm:implicit} (if $\Coop = \Op^\antishriek$ is the Koszul dual cooperad of a Koszul operad $\Op$, then $\Omega(\Coop) =\Op_\infty$). A key feature is that our formulas are expressed in terms of the compact description of $\Omega(\Coop)$-algebras as coderivation differentials on cofree $\Coop$-coalgebras. Another feature is that we obtain explicit formulas not only for the transferred $\Omega(\Coop)$-algebra structure, but for $\Omega(\Coop)$-structures on all maps in the contraction. A curious discovery is that the structures on the projection and the homotopy depend on the choice of pseudo-derivation extending the original homotopy, whereas the transferred structure and the structure on the inclusion do not, see Theorem \ref{thm:transfer}. This observation seems to be new even in the case of $A_\infty$-algebras.

We should point out that \emph{existence} of a transferred structure is well known and follows from general principles (see, e.g., \cite{BM,BV,JY,Markl}), essentially because operads of the form $\Omega(\Coop)$ are cofibrant. But such abstract considerations do not yield tractable explicit formulas. Explicit formulas for transferring $\Op_\infty$-algebra structures, for $\Op$ a Koszul operad, have been obtained independently by Loday-Vallette \cite[Theorem 10.3.3]{LV}. The advantage of our approach is that we obtain simple and transparent formulas in terms of the compact description of $\Op_\infty$-algebras as coderivation differentials on cofree $\Op^\antishriek$-coalgebras. The compact description is for many purposes the most convenient one to work with, and it is desirable to have a transfer theorem in this form. Furthermore, we can recover the Loday-Vallette formulas from our formulas, see Theorem \ref{thm:transfer}.

A perturbation lemma for cocommutative coalgebras, yielding transfer of $L_\infty$-algebra structures, has been obtained by Huebschmann \cite{Huebschmann-Lie,Huebschmann-sh} using different methods.

\subsection*{Statements of results}
Let us first introduce the two new notions: \emph{thick maps} and \emph{pseudo-derivations}.
\begin{definition}
Let $A$ and $B$ be chain complexes over a commutative ring $\kk$. A \defn{thick map $\thf\colon A\rightarrow B$} is a sequence of maps of the same degree $|\thf|$
$$\thf_n\colon A^{\tensor n}\rightarrow B^{\tensor n},\quad n\geq 0.$$
We say that $\thf$ is a \emph{symmetric thick map} if each $\thf_n$ is equivariant with respect to the action of the symmetric group $\Sigma_n$ permuting tensor factors. If $A$ and $B$ are algebras over an operad $\Op$ then we say that $\thf$ is a \emph{thick map of $\Op$-algebras} if $\thf_1\mu_A = (-1)^{|\thf||\mu|}\mu_B\thf_n$ for all $\mu\in \Op(n)$.
\end{definition}
Thick maps of $\Op$-algebras are a simultaneous generalization of morphisms and derivations. Indeed, morphisms of $\Op$-algebras may be identified with thick maps of $\Op$-algebras $\thf\colon A\rightarrow B$ that satisfy $\thf_{p+q} = \thf_p\tensor \thf_q$, and derivations may be identified with thick maps of $\Op$-algebras $\thd\colon A\rightarrow A$ that satisfy $\thd_{p+q} = \thd_p\tensor \tho + \tho\tensor \thd_q$, see Proposition \ref{prop:TO}. Levelwise composition, addition and differentiation of thick maps make $\Op$-algebras together with thick maps of $\Op$-algebras into a dg-category, i.e., a category enriched in chain complexes. Just as in any dg-category, a \defn{contraction} is a diagram
$$\Diag\colon \SDR{A}{B}{\thf}{\thg}{\thh},$$
where $|\thf|=|\thg|=0$, $|\thh| = 1$ and
$$\partial(\thf)=\thzr, \quad \partial(\thg)=\thzr, \quad \partial(\thh) = \thg\thf - \tho, \quad \thf\thg=\tho$$
$$\thf\thh = \thzr,\quad \thh\thh = \thzr,\quad \thh\thg = \thzr.$$

\begin{definition}[Contraction of $\Op$-algebras]
If $A$ and $B$ are $\Op$-algebras, then we say that a contraction $\Diag$ is a \defn{contraction of $\Op$-algebras} if $\thf$, $\thg$ are morphisms and if $\thh$ is a \defn{pseudo-derivation}, by which we mean that
$$(\thh_p\tensor \tho - \tho\tensor \thh_q) \thh_{p+q} = \thh_p\tensor \thh_q$$
and
$$\thh_{p+q}(\thh_p\tensor \tho - \tho\tensor \thh_q) = - \thh_p\tensor \thh_q$$
for all $p,q\geq 0$.
\end{definition}

When $\Op$ is the operad governing associative algebras, pseudo-derivations generalize algebra homotopies in the sense of \cite{GLS2,HK} (see Proposition \ref{prop:pseudo-derivations generalize derivations}), and the following theorem is a generalization of the `Algebra Perturbation Lemma' \cite[$(2.1^*)$]{HK}. A \emph{perturbation} of an $\Op$-algebra $A$ is a derivation $\tht\colon A\rightarrow A$ satisfying $\partial(\tht) + \tht^2 = \thzr$.

\begin{theorem}[$\Op$-algebra Perturbation Lemma] \label{thm:PL} Let $\Diag$ be a contraction of $\Op$-algebras. If $\tht$ is a perturbation of $A$ then, provided the series $\tht+\tht\thh\tht +\ldots$ converges, the recursive formulas
\begin{align*}
\thfp & = \thf + \thfp\tht\thh, & \thgp & = \thg + \thh\tht\thgp, \\
\thhp & = \thh + \thhp\tht\thh, & \thtp & = \thf\tht\thgp,
\end{align*}
define a perturbation $\thtp$ of $B$ and a contraction of $\Op$-algebras
$$\Diag^\tht \colon \bigSDR{(A,d_A+\tht_1)}{(B,d_B+\tht_1')}{\thfp}{\thgp}{\thhp}$$
In particular, $\thfp$, $\thgp$ are morphisms, $\thtp$ is a derivation and $\thhp$ is a pseudo-derivation.
\end{theorem}
There is also a dual of Theorem \ref{thm:PL} for coalgebras over cooperads, see Theorem \ref{thm:TCPL}. In practice, convergence is often ensured by having suitable filtrations on the objects.

Another interesting feature of thick maps is that they provide means of linearizing non-additive functors. More precisely, we show that the free $\Op$-algebra functor $\Op[A] = \bigoplus_{n\geq 0} \Op(n)\tensor_{\Sigma_n} A^{\tensor n}$ extends to a dg-functor $\Op_\bullet[-]$ from the dg-category of chain complexes and symmetric thick maps to the dg-category of $\Op$-algebras and symmetric thick maps of $\Op$-algebras, see Proposition \ref{prop:Schur}. A consequence of this is the following theorem, which is a generalization of the `Tensor Trick', see \cite{HK,GLS2}.
\begin{theorem}[$\Op$-algebra Tensor Trick] \label{thm:tt}
Consider a contraction of chain complexes
$$\Diag\colon\SDR{A}{B}{f}{g}{h}.$$
If $\thh$ is a symmetric pseudo-derivation such that $\thh_1 = h$ and $\partial(\thh)= \thg\thf-\tho$, $\thh\thh = \thzr$, $\thf\thh = \thzr$, $\thh\thg=\thzr$, where $\thf_n =f^{\tensor n}$ and $\thg_n = g^{\tensor n}$, then there is an induced contraction of $\Op$-algebras
$$\Op_\bullet[\Diag]\colon \bigSDR{\Op[A]}{\Op[B]}{\Op_\bullet[\thf]}{\Op_\bullet[\thg]}{\Op_\bullet[\thh]}.$$
If $\Op$ is a non-symmetric operad, then one may drop the condition that $\thh$ be symmetric. There is always a non-symmetric pseudo-derivation $\thh$ with the requisite properties, namely
$$\thh_n = \sum_{p+1+q = n} 1^{\tensor p}\tensor h\tensor (gf)^{\tensor q}.$$
If $\kk$ contains the rational numbers $\rationals$ as a subring then, with $\thh_n$ as above,
$$\thh_n^\Sigma = \frac{1}{n!}\sum_{\sigma\in \Symgp{n}} \sigma^{-1} \thh_n \sigma,$$
defines a symmetric pseudo-derivation $\thh^\Sigma$ with the requisite properties.
\end{theorem}
We also give the dual of Theorem \ref{thm:tt} for cooperads, see Theorem \ref{thm:cooperad tensor trick}. If we demand the existence of a symmetric pseudo-derivation with the requisite properties for \emph{any} given contraction, then the condition $\rationals\subseteq\kk$ is necessary, see Proposition \ref{prop:sym}.

\subsection*{Application: Transfer theorem}
Let $\Coop$ be a cooperad, which we assume to be \emph{connected} in the sense that $\Coop(0)=0$ and $\Coop(1) = \kk$. For a chain complex $A$, the `cofree $\Coop$-coalgebra' is defined as
$$\Coop[A] = \bigoplus_{n\geq 0} \Coop(n)\tensor_{\Sigma_n} A^{\tensor n}.$$
Elements of the $n$-th summand are said to be of \emph{weight $n$}. Let $\Omega(\Coop)$ denote the cobar construction on $\Coop$. An $\Omega(\Coop)$-algebra structure on a chain complex $A$ may be identified with a weight decreasing coderivation perturbation $t$ of $\Coop[A]$, see Section \ref{section:transfer}. The \emph{bar construction} of an $\Omega(\Coop)$-algebra $(A,t)$ is the $\Coop$-coalgebra
$$\BB(A,t) = (\Coop[A],d_{\Coop[A]}+t).$$
If $(A,t)$ and $(B,t')$ are $\Omega(\Coop)$-algebras, and if $f\colon A\rightarrow B$ is a morphism between the underlying chain complexes, then an \emph{$\Omega(\Coop)$-structure} on $f$ is a morphism of $\Coop$-coalgebras $F\colon \BB(A,t)\rightarrow \BB(B,t')$ whose linear part is identified with $f$. We will also call such an $F$ a \emph{lax morphism of $\Omega(\Coop)$-algebras}.

An important special case is when $\Coop$ is the Koszul dual cooperad of a Koszul operad $\Op$. Then $\Omega(\Coop)$-algebras are exactly $\Op_\infty$-algebras, or `strongly homotopy' $\Op$-algebras, and a chain map with an $\Omega(\Coop)$-structure is the same as an $\Op_\infty$-map. In the case when $\Op$ is the associative operad, then the above amounts to the familiar compact definition of an $A_\infty$-algebra as a graded $\kk$-module $A$ together with a coderivation differential on the tensor coalgebra $T^c(sA)$. Other familiar examples of this form are $C_\infty$, $L_\infty$ or $G_\infty$-algebras. If the operad $\Op$ is not necessarily Koszul, one may define strongly homotopy $\Op$-algebras as algebras over the operad $\Omega  B \Op$, where the cooperad $B \Op$ is the bar construction on $\Op$.

\begin{theorem}[Transfer theorem; compact form] \label{thm:implicit}
Let $\Coop$ be a connected cooperad. Given a contraction of chain complexes
$$\SDR{A}{B}{f}{g}{h}$$
and an $\Omega(\Coop)$-algebra structure $t$ on $A$, there are explicit formulas for an $\Omega(\Coop)$-algebra structure $t'$ on $B$ and $\Omega(\Coop)$-structures $F',G',H'$ on $f,g,h$ that make
\begin{equation}
\bigSDR{\BB(A,t)}{\BB(B,t')}{F'}{G'}{H'}
\end{equation}
into a contraction of $\Coop$-coalgebras. The formulas are given by
\begin{align*}
F' & = F + FtH + F(tH)^2 + \cdots, & G' & = G + HtG + (Ht)^2G + \cdots, \\
H' & = H + HtH + H(tH)^2 + \cdots, & t' & = FtG +FtHtG + Ft(Ht)^2G + \cdots,
\end{align*}
where the maps
$$\SDR{\Coop[A]}{\Coop[B]}{F}{G}{H}$$
are defined by letting $F$, $G$ be the morphisms of $\Coop$-coalgebras induced by $f$, $g$. There are different possible choices for the homotopy $H$: For every choice of symmetric pseudo-derivation $\thh\colon A\rightarrow A$ that extends $h$ and satisfies
$$
\partial(\thh) = \thg\thf - \tho,\quad \thf\thh = \thzr,\quad \thh\thg =\thzr,\quad \thh\thh=\thzr,
$$
where $\thf_n = f^{\tensor n}$ and $\thg_n = g^{\tensor n}$, we may take $H = \Coop[\thh]$. If $\Coop$ is a non-symmetric cooperad, then one may drop the condition that $\thh$ be symmetric, and a possible choice of pseudo-derivation is
$$\thh_n = \sum_{p+1+q=n} 1^{\tensor p}\tensor h \tensor (gf)^{\tensor q}.$$
If $\rationals\subseteq \kk$ then a possible choice of symmetric pseudo-derivation is the symmetrization of $\thh_n$ above:
$$\thh_n^\Sigma = \frac{1}{n!}\sum_{\sigma\in \Symgp{n}} \sigma^{-1} \thh_n \sigma.$$
\end{theorem}

\begin{remark}
Convergence of the formulas is ensured because $t$ decreases weight. The hard part of the theorem is to show that $t'$ becomes a $\Coop$-coderivation and that $F'$ and $G'$ become morphisms of $\Coop$-coalgebras. The key point is that it is exactly the pseudo-derivation property that ensures this, and furthermore that it is always possible to find a suitable pseudo-derivation.
\end{remark}

\subsection*{Expanded form of the transfer theorem}
An $\Omega(\Coop)$-algebra structure $t$ on a chain complex $A$ may alternatively be described as a family of maps $t^\nu\colon A^{\tensor n} \rightarrow A$ of degree $|\nu|-1$, indexed by elements $\nu\in\Coop(n)$, satisfying certain relations, see Section \ref{section:transfer} for details. Similarly, if $(A,t)$ and $(B,t')$ are $\Omega(\Coop)$-algebras, then an $\Omega(\Coop)$-structure on a chain map $f\colon A\rightarrow B$ may be described as a family of maps $f^\nu\colon A^{\tensor n}\rightarrow B$ of degree $|\nu|$, indexed by elements $\nu\in \Coop(n)$, such that $f^1 = f$, subject to certain relations, see Section \ref{section:transfer}. In the case when $\Coop$ is the Koszul dual cooperad of the associative operad, the above simply amounts to the description of an $A_\infty$-algebra as a chain complex $A$ together with a family of maps $m_n\colon A^{\tensor n} \rightarrow A$, $n=2,3,\ldots$, satisfying the familiar relations.

The formulas in Theorem \ref{thm:implicit} may be expanded to recursive formulas expressed in terms of this alternative description of $\Omega(\Coop)$-algebras. To state them we need to introduce some notation. For $\nu\in \Coop(n)$ where $n\geq 2$ we will write the coproduct as
$$\Delta(\nu) = \nu\circ 1^{\tensor n} + 1\circ \nu + \sum_{q=1}^p \nu^q\circ (\nu_1^q\tensor \ldots \tensor \nu_{r_q}^q)\sigma_q\in (\Coop\circ\Coop)(n)$$
where $\nu^q$ and $\nu_i^q$ are elements of $\Coop$ of arity $<n$ and where $\sigma_q\in\Sigma_n$. Furthermore, we let
$$\Delta_{(1)}(\nu) = \sum_{i=1}^u (\nu_i'\circ_{e_i} \nu_i'')\tau_i$$
denote the \emph{quadratic part} of the coproduct, see Section \ref{section:operads}.

\begin{theorem}[Transfer Theorem; expanded form] \label{thm:transfer}
With notation as in Theorem \ref{thm:implicit}, we have the following recursive formulas for the transferred $\Omega(\Coop)$-algebra structure $t'$ on $B$ and for the $\Omega(\Coop)$-structures $F',G',H'$ on $f,g,h$:

For $\nu\in\Coop(n)$ where $n\geq 2$, we have
\begin{align*}
(t')^\nu & = f t^\nu g^{\tensor n} + \sum_{q=1}^p f t^{\nu^q} (g^{\nu_1^q}\tensor \ldots \tensor g^{\nu_{r_q}^q})\sigma_q,\\
g^\nu & = h t^\nu g^{\tensor n} + \sum_{q=1}^p h t^{\nu^q} (g^{\nu_1^q}\tensor \ldots \tensor g^{\nu_{r_q}^q})\sigma_q, \\
f^\nu & = (-1)^{|\nu|} f t^\nu \HH{n} + \sum_{i=1}^u (-1)^{|\nu_i''|} (f^{\nu_i'}\circ_{e_i} t^{\nu_i''})\tau_i \HH{n},\\
h^\nu & = (-1)^{|\nu|} h t^\nu \HH{n} + \sum_{i=1}^u (-1)^{|\nu_i''|} (h^{\nu_i'}\circ_{e_i} t^{\nu_i''})\tau_i \HH{n}.
\end{align*}
In particular, $t'$ and $G'$ do not depend on the choice of pseudo-derivation $\thh$ extending $h$. However, $F'$ and $H'$ do depend on this choice.

\end{theorem}
These recursive formulas may be interpreted as tree formulas. In Section 12 we explain this point of view in detail in the special case of $A_\infty$-algebras. In fact, in that case we recover exactly the formulas written down by Kontsevich-Soibelman \cite[\S 6.4]{KS}. See also \cite{Huebschmann-A}. In the case when $\Coop$ is the Koszul dual cooperad of a Koszul operad, similar considerations show more generally that the structure we obtain agrees with the one in Loday-Vallette \cite[Theorem 10.3.3]{LV}.

If the ground ring $\kk$ is a field then it is always possible to find a contraction between a chain complex $A$ and its homology $\H_*(A)$. Therefore, the following is a corollary to Theorem \ref{thm:transfer}.
\begin{corollary}[Minimality theorem for $\Omega(\Coop)$-algebras]
Suppose that $\kk$ is a field of characteristic zero and let $\Coop$ be a connected cooperad. Let $A$ be a chain complex with an $\Omega(\Coop)$-algebra structure $t$. Then there exist an $\Omega(\Coop)$-algebra structure $t'$ on the homology $\H_*(A)$, with trivial differential, and a lax contraction of $\Omega(\Coop)$-algebras
$$\bigSDR{(A,t)}{(\H_*(A),t')}{f^\bullet}{g^\bullet}{h^\bullet}.$$
In particular, every $\Omega(\Coop)$-algebra $(A,t)$ is quasi-isomorphic to a minimal $\Omega(\Coop)$-algebra $(\H_*(A),t')$. If $\Coop$ is a non-symmetric operad, then one may drop the assumption that $\kk$ is of characteristic zero.
\end{corollary}

\subsection*{Outline of the paper}
In Section \ref{section:background} we review the relevant background material on homological perturbation theory. In Section \ref{section:thick maps} we introduce some machinery for handling thick maps. The proofs of Theorem \ref{thm:PL} and Theorem \ref{thm:tt} are both separated into two parts, a first part dealing with formal properties of thick maps without reference to any operad, and a second part where the operad enters. Section \ref{section:thick contractions} contains the first part of the proof of Theorem \ref{thm:PL}. In it, we introduce and study pseudo-derivations and thick contractions. The first part of the proof of Theorem \ref{thm:tt} is contained in Section \ref{section:symmetric tensor trick} where we show how to extend any contraction to a symmetric thick contraction. Section \ref{section:operads} contains a review of the basic definitions concerning operads and cooperads that we will use. The second part of the proof of Theorem \ref{thm:PL} is contained in Section \ref{section:PL}. In 
Section \ref{section:TT} we extend the free $\Op$-algebra functor to the dg-category of thick maps, and we use this to finish the proof of Theorem \ref{thm:tt}. In Section \ref{section:cooperads} we define thick maps of $\Coop$-coalgebras, where $\Coop$ is a cooperad, and we give the duals of Theorem \ref{thm:PL} and Theorem \ref{thm:tt}. In Section \ref{section:maps} we prove a general result about thick maps between `cofree' $\Coop$-coalgebras which is used in the proof of Theorem \ref{thm:transfer}. In Section \ref{section:transfer} we give the proof of Theorem \ref{thm:implicit} and Theorem \ref{thm:transfer}. In Section \ref{section:tree} we illustrate how the formulas in Theorem \ref{thm:transfer} work in the case of $A_\infty$-algebras.

\subsection*{Conventions}
In this paper, the term `chain complex' will mean unbounded chain complex over a commutative ground ring $\kk$. The differential of a chain complex $A$ will be denoted by $d_A$ and we take it to be of degree $-1$. Recall that a \defn{dg-category} is a category $\Acat$ enriched over chain complexes, i.e., a collection of objects $\Ob \Acat$ and for every two objects $A$ and $B$ a chain complex $\Hom_\Acat(A,B)$, elements of which we will refer to as \defn{maps} from $A$ to $B$, together with natural composition and unit morphisms that satisfy standard unit and associativity axioms, see for instance \cite{Keller-dg}. We will use $\partial$ as a generic notation for the differential in $\Hom_\Acat(A,B)$. Thus, in the dg-category $\Cat$ of chain complexes, $\partial(f) = d_B f - (-1)^{|f|}fd_A$ for maps $f\in\Hom_{\Cat}(A,B)$.

\section{Background on homological perturbation theory} \label{section:background}
In this section we will review some of the classical results of homological perturbation theory. The central notion, which goes back to Eilenberg and MacLane \cite[\S 12]{EMI}, is that of a contraction.
\begin{definition}
A \defn{contraction} is a diagram of maps of chain complexes
$$\Diag\colon \SDR{A}{B}{f}{g}{h},$$
where $|f|=|g|=0$, $|h| = 1$, $\partial(f)=0$, $\partial(g)=0$, and
$$\partial(h) = gf - 1_A,\quad\quad fg=1_B.$$

Furthermore, we impose the annihilation conditions
$$fh = 0, \quad hh = 0, \quad hg = 0.$$
We say that $\Diag$ is a \defn{filtered contraction} if $A$ and $B$ are equipped with bounded below exhaustive filtrations which are preserved by the maps $f$, $g$ and $h$.
\end{definition}

In plain English, $f$ and $g$ are morphisms of chain complexes with $fg=1_B$ and $h$ is a chain homotopy from $gf$ to $1_A$. Thus, $B$ is a strong deformation retract of $A$. For this reason, the term `SDR-data' is often used as an alternative to `contraction'.

\begin{remark}
It is harmless to assume the annihilation conditions, as was pointed out in \cite{LamS}. If they are not satisfied, then one can replace $h$ by $h'' = -h'dh'$, where $h' = \partial(h)h\partial(h)$, to get a contraction.
\end{remark}

A \defn{perturbation} of $A$ is a map $t\colon A\rightarrow A$ of degree $-1$ such that $\partial(t) + t^2 = 0$, or, equivalently, $(d_A+t)^2 = 0$. Let $A^t$ denote the chain complex $A$ endowed with the new differential $d_A + t$. The following result is the basis for the theory.
\begin{theorem}[Basic Perturbation Lemma, \cite{Brown,Gugenheim}] \label{thm:BPL}
If $t$ is a perturbation of $A$ such that $1-ht$ is invertible then setting $\Sigma = t(1-ht)^{-1}$ the following formulas define a perturbation $t'$ of $B$ and a new contraction
\begin{equation*}
\Diag^t\colon \SDR{A^t}{B^{t'}}{f'}{g'}{h'},\quad\quad \begin{array}{ll} f' = f + f\Sigma h, &
g' = g + h\Sigma g, \\
h' = h + h\Sigma h, &
t' = f\Sigma g.\end{array}
\end{equation*}
\end{theorem}

\begin{remark}
In the original statement of the Basic Perturbation Lemma \cite{Gugenheim} one assumes that $\Diag$ is a filtered contraction and that the perturbation $t$ lowers the filtration on $A$. Then the infinite series $\sum_{n\geq 0} (ht)^n$ converges pointwise and defines an inverse of $1-ht$. It was observed in \cite{BL} that invertibility of $1-ht$ is a sufficient hypothesis. Observe also that invertibility of $1-ht$ is equivalent to invertibility of $1-th$. Indeed, $(1-th)^{-1} = 1+t(1-ht)^{-1}h$.
\end{remark}

\begin{definition}[\cite{GLS2,HK}] \label{def:algebra contraction}
A \defn{contraction of algebras} is a contraction $\Diag$ where $A$ and $B$ are differential graded algebras i.e., chain complexes equipped with morphisms $\mu_A\colon A\tensor A\rightarrow A$ and $\mu_B\colon B\tensor B\rightarrow B$, where $f$ and $g$ are morphisms of algebras and where $h$ is an \defn{algebra homotopy}, which means that 
$$h\mu_A = \mu_B(h\tensor gf + 1\tensor h).$$
\end{definition}

\begin{theorem}[Algebra Perturbation Lemma, {\cite[\S 2.2]{GLS2}}, {\cite[$(2.1^*)$]{HK}}] \label{thm:APL}
If $\Diag$ is a contraction of algebras and if the perturbation $t$ is a derivation, i.e., $t\mu_A = \mu_A(t\tensor 1 + 1\tensor t)$, then $\Diag^t$ is a contraction of algebras.
\end{theorem}

The `Tensor Trick' is a way of producing an algebra contraction starting from any contraction. Recall that the \defn{tensor algebra} on a chain complex $A$ is the chain complex
$$\Talg(A) = \bigoplus_{n\geq 0} A^{\tensor n}$$
with multiplication $\mu\colon \Talg(A)\tensor \Talg(A)\rightarrow \Talg(A)$ induced by the canonical isomorphisms $A^{\tensor p}\tensor A^{\tensor q} \cong A^{\tensor (p+q)}$.
\begin{theorem}[Tensor Trick, {\cite[\S 3.2]{GL}},{\cite[\S 3]{GLS2}},{\cite[$(2.2.0^*)$]{HK}}] \label{thm:tensor trick}
For any contraction $\Diag$ the following is a contraction of algebras
$$\Talg(\Diag)\colon \SDR{\Talg(A)}{\Talg(B)}{F}{G}{H},$$
where $F$, $G$ and $H$ act on tensors of length $n$ by, respectively,
$$f^{\tensor n}, \quad g^{\tensor n}, \quad \sum_{i+1+j=n} 1^{\tensor i}\tensor h \tensor (gf)^{\tensor j}.$$
\end{theorem}

As remarked in \cite[Remark, end of \S 2.2]{GLS2}, if $\mu_A$ is a commutative operation, then the left hand side of the equation
$$h\mu_A  = \mu_A(h\tensor gf + 1\tensor h)$$
is symmetric but the right hand side is not. For this reason, the present notion of an algebra homotopy is not useful for commutative algebras or, more generally, for algebras where symmetries play a role. In what follows, we will look for a symmetric generalization of the notion of an algebra homotopy such that Theorem \ref{thm:APL} and Theorem \ref{thm:tensor trick}, appropriately modified, remain valid.

\section{Thick maps} \label{section:thick maps}
\begin{definition} \label{def:thick map}
Let $A$ and $B$ be chain complexes. We define a \defn{thick map} $\thf \colon A\rightarrow B$ to be a sequence of maps
$$\thf  = \{\thf_n\colon A^{\tensor n}\rightarrow B^{\tensor n}\}_{n\geq 0}$$
of the same degree $|\thf|$. We say it is \defn{symmetric} if each $\thf_n$ is equivariant with respect to the action of the symmetric group $\Symgp{n}$ permuting tensor factors.
\end{definition}

There is a dg-category $\Th{\Cat}$ of thick maps. It has the same objects as the dg-category $\Cat$ of chain complexes but $\Hom_{\Th{\Cat}}(A,B)$ is the chain complex of thick maps from $A$ to $B$. The $\kk$-linear structure, differentials and compositions are defined by
\begin{align*}
(a\thf + b\thh)_n & = a\thf_n + b\thh_n,\\
\partial(\thf)_n & = d_{B^{\tensor n}} \thf_n - (-1)^{|\thf|}\thf_n d_{A^{\tensor n}}, \\
(\thg\circ \thf)_n & = \thg_n\circ \thf_n,
\end{align*}
for $\thf,\thh\colon A\rightarrow B$, $\thg\colon B\rightarrow C$, $a,b\in\kk$, and where $d_{A^{\tensor n}}$ is the ordinary tensor product differential on $A^{\tensor n}$. Chain complexes together with \emph{symmetric} thick maps form a dg-subcategory $\Ths{\Cat}$ of $\Th{\Cat}$. The \emph{identity} $\tho\colon A\rightarrow A$ and the \emph{zero map} $\thzr\colon A\rightarrow B$ are the thick maps with $\tho_n = 1_{A^{\tensor n}}$ and $\thzr_n = 0$. We will now give names to thick maps with special properties.

\begin{definition}
\begin{enumerate}
\item We say that a thick map $\thf \colon A\rightarrow B$ is a \defn{morphism} if $\thf_{p+q} = \thf_p\tensor \thf_q$ for all $p,q\geq 0$.
\item Let $\thl$ and $\thr$ be morphisms from $A$ to $B$. We say that a thick map $\thd\colon A\rightarrow B$ is an \defn{$(\thl,\thr)$-derivation} if $\thd_{p+q} = \thd_p\tensor \thr_q + \thl_p\tensor \thd_q$ for all $p,q\geq 0$.
\item For simplicity, a $(\tho,\tho)$-derivation $\thd\colon A\rightarrow A$ will be called a \defn{derivation}.
\end{enumerate}
\end{definition}

Let us also introduce a notational device. If $\thf \colon A\rightarrow B$ and $\thg \colon C\rightarrow D$ are two thick maps, we can form the bi-indexed sequence
$$\thf \tensor \thg = \{\thf_p\tensor \thg_q\colon A^{\tensor p}\tensor C^{\tensor q}\rightarrow B^{\tensor p}\tensor D^{\tensor q}\}_{p,q\geq 0}.$$
We can also form the bi-indexed sequence
$$m^*(\thf) = \{\thf_{p+q}\colon A^{\tensor p}\tensor A^{\tensor q}\rightarrow B^{\tensor p}\tensor B^{\tensor q}\}_{p,q\geq 0}.$$
Then it is clear that a thick map $\thf \colon A\rightarrow B$ is a morphism if and only if $m^*(\thf) = \thf\tensor \thf$
and that a thick map $\thd\colon A\rightarrow B$ is an $(\thl,\thr)$-derivation if and only if $m^*(\thd) = \thd\tensor\thr + \thl\tensor\thd$.

\section{Thick contractions} \label{section:thick contractions}
Using thick maps we can reformulate the notion of an algebra contraction in a way that lends itself to generalizations. By a \defn{thick contraction} we mean a contraction in the dg-category $\Th{\Cat}$.
\begin{proposition} \label{prop:algebra contraction}
Any contraction $\Diag$ has a unique extension to a thick contraction
$$\thDiag \colon \SDR{A}{B}{\thf}{\thg}{\thh}$$
where $\thf$ and $\thg$ are morphisms and $\thh$ is a $(\tho,\thg\thf)$-derivation. Furthermore, if $A$ and $B$ are algebras then $\Diag$ is an algebra contraction if and only if $\thf$, $\thg$ and $\thh$ are compatible with the algebraic structure in the sense that
$$\thf_1\mu_A = \mu_B \thf_2, \quad \thg_1\mu_B = \mu_A \thg_2, \quad \thh_1\mu_A = \mu_A \thh_2.$$
\end{proposition}

\begin{proof}
Requiring that $\thf$, $\thg$ are morphisms and that $\thh$ is a $(\tho,\thg\thf)$-derivation leaves us with no choice but to set
$$\thf_n = f^{\tensor n}, \quad \thg_n = g^{\tensor n}, \quad \thh_n = \sum_{i+1+j=n} 1^{\tensor i}\tensor h \tensor (gf)^{\tensor j}.$$
But these formulas coincide with the formulas in the Tensor Trick (Theorem \ref{thm:tensor trick}), and it is a consequence of that theorem that they define a thick contraction. Next, $\Diag$ is an algebra contraction (Definition \ref{def:algebra contraction}) if and only if
$$fm_A = m_Bf^{\tensor 2},\quad g\mu_B = \mu_A g^{\tensor 2}, \quad h\mu_A = \mu_A (h\tensor gf + 1\tensor h).$$
In view of our definition of $\thf$, $\thg$ and $\thh$, these conditions are the same as the conditions in the statement of the proposition.
\end{proof}

We repeat that the problem with algebra homotopies is the asymmetry in the expression $h\tensor gf + 1\tensor h$. In other words, the problem is that if a thick map $\thh$ is a $(\tho,\thg\thf)$-derivation, then it can hardly be symmetric in the sense of Definition \ref{def:thick map}. The goal for the remainder of this section is the following:
\emph{Generalize the condition `$\thh$ is a $(\tho,\thg\thf)$-derivation' to a condition that makes sense for symmetric thick maps.}
There are two constraints:
\begin{itemize}
\item The condition should be sufficiently close to the $(\tho,\thg\thf)$-derivation condition so that the proof of the Algebra Perturbation Lemma goes through.
\item The condition should be flexible enough so as to allow for a `symmetric tensor trick', i.e., an extension of any contraction to a \emph{symmetric} thick contraction which satisfies the condition.
\end{itemize}

We will argue that the following definition contains the solution to this problem.
\begin{definition} \label{def:pseudo-derivation}
A thick map $\thh\colon A\rightarrow A$ is a \defn{pseudo-derivation} if
$$(\thh\tensor \tho - \tho\tensor \thh)m^*(\thh) = \thh\tensor\thh,$$
and
$$m^*(\thh)(\thh\tensor \tho - \tho\tensor \thh) = -\thh\tensor \thh.$$
In other words, $\thh$ is a pseudo-derivation if for all $p,q\geq 0$
$$(\thh_p\tensor \tho - \tho\tensor \thh_q)\thh_{p+q} = \thh_p\tensor \thh_q$$
and
$$\thh_{p+q}(\thh_p\tensor \tho - \tho \tensor \thh_q) = -\thh_p\tensor \thh_q.$$
\end{definition}
For the rest of the section, fix a thick contraction
$$\thDiag\colon \SDR{A}{B}{\thf}{\thg}{\thh}.$$
To begin with, let us note that pseudo-derivations generalize $(\tho,\thg\thf)$-derivations.
\begin{proposition} \label{prop:pseudo-derivations generalize derivations}
If the homotopy $\thh$ in $\thDiag$ is a $(\tho,\thg\thf)$-derivation then $\thh$ is a pseudo-derivation.
\end{proposition}

\begin{proof}
This is a simple calculation:
\begin{align*}
(\thh\tensor \tho - \tho\tensor \thh)m^*(\thh) & = (\thh\tensor \tho - \tho\tensor \thh)(\tho\tensor \thh + \thh \tensor \thg\thf) \\
& = \thh\tensor \thh - \tho\tensor \thh\thh + \thh\thh\tensor \thg\thf + \thh \tensor \thh\thg\thf \\
&= \thh\tensor \thh.
\end{align*}
Here we have used the annihilation conditions $\thh\thh = \thzr$ and $\thh\thg = \thzr$. Similarly, one verifies that $-m^*(\thh)(\thh\tensor \tho - \tho\tensor \thh) = \thh\tensor \thh$.
\end{proof}

Fix a thick perturbation $\tht$ of $A$, i.e., a thick map of degree $-1$ satisfying $\partial(\tht) + \tht^2 = \thzr$. Suppose that $\tho-\thh\tht$ (or equivalently $\tho-\tht\thh$) is invertible, so that we can use the formulas of the Basic Perturbation Lemma (Theorem \ref{thm:BPL}) to define thick maps $\thfp$, $\thgp$, $\thhp$, $\thtp$. The following theorem, which shows that the pseudo-derivation property is sufficient to make the Algebra Perturbation Lemma work, is the main result of this section.
\begin{theorem} \label{thm:thick perturbation}
Let $\thDiag$ be a thick contraction. If $\thf$ and $\thg$ are morphisms, $\thh$ a pseudo-derivation and $\tht$ a derivation, then $\thfp$ and $\thgp$ are morphisms, $\thhp$ a pseudo-derivation, $\thtp$ a derivation, $t = \tht_1$ and $t' = \tht_1'$ are perturbations of $A$ and $B$, respectively, and
$$\thDiag^{\tht}\colon \SDR{A^{t}}{B^{t'}}{\thfp}{\thgp}{\thhp}$$
is a thick contraction. Furthermore, if $\thh$ is symmetric, then so is $\thhp$.
\end{theorem}
The proof of this theorem will occupy the rest of the section.
\vskip8pt
\begin{proposition}
If $\thh$ is a pseudo-derivation and $\tht$ is a derivation then $\thhp$ is a pseudo-derivation.
\end{proposition}

\begin{proof}
We need to show that $(\thhp\tensor \tho - \tho\tensor \thhp)m^*(\thhp) = -m^*(\thhp)(\thhp\tensor \tho - \tho\tensor \thhp) = \thhp\tensor \thhp$. If we multiply the right hand side from the left with $(\tho-\thh \tht)\tensor (\tho-\thh \tht)$ and from the right with $m^*(\tho-\tht\thh)$ and use that $(\tho-\thh \tht)\thhp = \thhp (\tho-\tht\thh) = \thh$ we get
\begin{align*}
\big((\tho-\thh \tht)\tensor (\tho-\thh \tht)\big)&  (\thhp\tensor \tho - \tho\tensor \thhp)m^*(\thhp)m^*(\tho-\tht\thh) \\
& = (\thh\tensor (\tho-\thh\tht) - (\tho-\thh\tht)\tensor \thh)m^*(\thh) \\
& = (\thh\tensor \tho - \tho\tensor \thh)m^*(\thh) - (\thh\tensor \thh)(\tht\tensor \tho + \tho\tensor \tht)m^*(\thh) \\
& = \thh\tensor \thh - (\thh\tensor \thh)m^*(\tht\thh) \\
& = (\thh\tensor \thh)m^*(\tho-\tht\thh) \\
& = \big((\tho-\thh\tht)\tensor (\tho-\thh\tht)\big)(\thhp\tensor \thhp)m^*(\tho-\tht\thh).
\end{align*}
Since $(\tho-\thh\tht)$ and $(\tho-\tht\thh)$ are invertible, the above equation implies that
$$(\thhp\tensor \tho - \tho\tensor \thhp)m^*(\thhp) = \thhp\tensor \thhp.$$
Similarly one verifies that $-m^*(\thhp)(\thhp\tensor \tho - \tho\tensor \thhp) = \thhp\tensor \thhp$.
\end{proof}

We will see in Proposition \ref{prop:subsume} below that the hypotheses in Theorem \ref{thm:thick perturbation} imply the following additional conditions:

\noindent {\bf Module conditions.}
\begin{align*}
(\thf\tensor \tho)m^*(\thh) & = \thf\tensor \thh & m^*(\thh)(\thg\tensor \tho) & = \thg\tensor \thh \\
(\tho\tensor \thf)m^*(\thh) & = \thh\tensor \thf & m^*(\thh)(\tho\tensor \thg) & = \thh\tensor \thg
\end{align*}
The module conditions together with the pseudo-derivation property are exactly what we need to ensure that $\thfp$ and $\thgp$ are morphisms and that $\thtp$ is a derivation provided that $\thf$ and $\thg$ are morphisms and $\tht$ is a derivation.
\begin{proposition} \label{prop:stability}
Suppose that $\thh$ is a pseudo-derivation, that the module conditions are satisfied and that $\tht$ is a derivation.
\begin{enumerate}
\item If $\thf$ is a morphism then so is $\thfp$. \label{f}
\item If $\thg$ is a morphism then so is $\thgp$. \label{g}
\item If $\thf$ and $\thg$ are morphisms then $\thtp$ is a derivation. \label{t}
\end{enumerate}
\end{proposition}

\begin{proof}
\eqref{f}: We need to verify that $m^*(\thfp) = \thfp\tensor \thfp$ under the assumption $m^*(\thf) = \thf\tensor\thf$. Observe that $\thfp = \thf + \thfp\tht\thh$. Therefore,
\begin{align*}
(\thfp \tensor \thfp )m^*(\tht \thh) & = (\thfp \tensor \thfp )(\tht \tensor \tho + \tho\tensor \tht )m^*(\thh) \\
& = (\thfp \tht  \tensor (\thf  + \thfp \tht \thh) + (\thf  + \thfp \tht \thh)\tensor \thfp \tht )m^*(\thh) \\
& = (\thfp \tht \tensor \tho)(\tho\tensor\thf)m^*(\thh) + (\tho\tensor \thfp\tht)(\thf\tensor \tho)m^*(\thh) \\
& \quad - (\thfp \tht \tensor \thfp \tht)(\thh\tensor \tho - \tho\tensor \thh)m^*(\thh) \\
& = (\thfp \tht \tensor \tho)(\thh\tensor \thf) + (\tho\tensor \thfp\tht)(\thf\tensor \thh) - (\thfp \tht\tensor \thfp \tht)(\thh \tensor \thh)\\
& = \thfp \tht\thh \tensor \thf  + \thf \tensor \thfp \tht\thh + \thfp \tht\thh\tensor \thfp \tht\thh \\
& = (\thfp - \thf)\tensor \thf + \thf\tensor (\thfp - \thf) + (\thfp - \thf)\tensor (\thfp - \thf) \\
& = \thfp \tensor \thfp  - \thf \tensor \thf.
\end{align*}
Here we have used that $\thh$ is a pseudo-derivation, that $\tht$ is a derivation and the module conditions involving $\thf$. The above gives that
$$
(\thfp \tensor \thfp )m^*(\tho-\tht\thh) = \thf \tensor \thf = m^*(\thf ) = m^*(\thfp (\tho-\tht\thh)) = m^*(\thfp )m^*(\tho-\tht\thh),
$$
and this implies that $\thfp \tensor \thfp  = m^*(\thfp)$ since $\tho-\tht\thh$ is invertible. 

\eqref{g}: This is proved as \eqref{f} but uses the module conditions involving $\thg$ instead.

\eqref{t}: Note that $\thtp = \thfp \tht\thg$. Since $\thh\thg = \thzr$ and $\thfp(\tho-\tht\thh) = \thf$, we have that $\thfp\thg = \thfp(\tho-\tht\thh)\thg = \thf\thg = \tho$. By \eqref{f}, $\thfp$ is a morphisms. Combining these facts we get that
\begin{align*}
m^*(\thtp) & = m^*(\thfp)m^*(\tht)m^*(\thg) = (\thfp\tensor \thfp)(\tht\tensor \tho + \tho \tensor \tht)(\thg\tensor \thg) \\
& = \thfp\tht\thg\tensor \thfp\thg + \thfp\thg\tensor \thfp\tht\thg = \thtp\tensor \tho + \tho\tensor \thtp,
\end{align*}
so $\thtp$ is indeed a derivation.
\end{proof}

To show that the module conditions are satisfied under the hypotheses of Theorem \ref{thm:thick perturbation}, we will introduce an auxiliary set of conditions on $\thDiag$, called the `annihilation conditions', summarized as follows: all possible ways of forming maps $m^*(\thx)(\thy\tensor \thz)$ or $(\thx\tensor \thy)m^*(\thz)$ where $\{\thh\}\subseteq \{\thx,\thy,\thz\}\subseteq \{\thf,\thg,\thh\}$ should yield the zero map.

\noindent {\bf Annihilation conditions.}
\begin{align*}
(\thh\tensor \thh)m^*(\thh) & = \thzr, & m^*(\thh)(\thh\tensor \thh) & = \thzr, \\
(\thh\tensor \thh)m^*(\thg) & = \thzr, & m^*(\thf)(\thh\tensor \thh) & = \thzr, \\
(\thf\tensor \thh)m^*(\thh) & = \thzr, & m^*(\thh)(\thg\tensor \thh) & = \thzr, \\
(\thh\tensor \thf)m^*(\thh) & = \thzr, & m^*(\thh)(\thh\tensor \thg) & = \thzr, \\
(\thf\tensor \thf)m^*(\thh) & = \thzr, & m^*(\thh)(\thg\tensor \thg) & = \thzr, \\
(\thf\tensor \thh)m^*(\thg) & = \thzr, & m^*(\thf)(\thg\tensor \thh) & = \thzr, \\
(\thh\tensor \thf)m^*(\thg) & = \thzr, & m^*(\thf)(\thh\tensor \thg) & = \thzr.
\end{align*}
The annihilation conditions, albeit outnumbering the module conditions, are easier to verify, and, getting ahead of ourselves, we will take advantage of this in proving Theorem \ref{thm:symmetric tensor trick}.

\begin{proposition} \label{prop:annihilation conditions}
\begin{enumerate}
\item The homotopy $\thh$ is a pseudo-derivation if and only if the annihilation conditions in the four first rows are satisfied. \label{pd}
\item The module conditions are equivalent to the annihilation conditions in the five last rows. \label{md}
\end{enumerate}
\end{proposition}

\begin{proof}
\eqref{pd}: Consider the differential of the map $(\thh\tensor \thh)m^*(\thh)$:
\begin{align*}
\partial((&\thh\tensor \thh)m^*(\thh)) \\
& = ((\thg\thf-\tho)\tensor \thh)m^*(\thh) - (\thh\tensor (\thg\thf - \tho))m^*(\thh) + (\thh\tensor \thh)m^*(\thg\thf-\tho) \\
& = (\thh\tensor \tho - \tho\tensor \thh)m^*(\thh) - \thh\tensor \thh  \\
& \quad + (\thg\tensor \tho)(\thf\tensor \thh)m^*(\thh) - (\tho\tensor \thg)(\thh\tensor \thf)m^*(\thh) - (\thh\tensor \thh)m^*(\thg)m^*(\thf).
\end{align*}
From this expression, one sees that the equality $(\thh\tensor \tho - \tho\tensor \thh)m^*(\thh) = \thh\tensor \thh$ follows from the first four annihilation conditions in the left column. Conversely, these four annihilation conditions follow from $(\thh\tensor \tho - \tho\tensor \thh)m^*(\thh) = \thh\tensor \thh$:
$$(\thh\tensor \thh)m^*(\thh) = (\thh\tensor \tho - \tho \tensor \thh)m^*(\thh)m^*(\thh) = (\thh\tensor \tho - \tho \tensor \thh)m^*(\thh\thh) = \thzr,$$
and similarly $(\thh\tensor \thh)m^*(\thg) = \thzr$. Next,
\begin{align*}
(\thf\tensor \thh)m^*(\thh) & = (\thf\tensor \tho)(\tho\tensor \thh)m^*(\thh) = (\thf\tensor \tho)((\thh\tensor \tho)m^*(\thh) - \thh\tensor \thh) 
\\ & = (\thf\thh\tensor \tho)m^*(\thh) - \thf\thh\tensor\thh = \thzr,
\end{align*}
and similarly $(\thh\tensor \thf)m^*(\thh) = \thzr$. The condition $-m^*(\thh)(\thh\tensor\tho - \tho\tensor \thh) = \thh\tensor \thh$ is likewise equivalent to the first four annihilation conditions in the right column.

\eqref{md}: By the same token, each individual module condition is equivalent to three annihilation conditions. The module condition $(\thf\tensor \tho)m^*(\thh) = \thf\tensor \thh$ is equivalent to the three annihilation conditions
$$(\thf\tensor \thh)m^*(\thh) = \thzr, \quad (\thf\tensor \thh)m^*(\thg) = \thzr,\quad (\thf\tensor \thf)m^*(\thh) = \thzr.$$
The proof is similar to the proof of \eqref{pd} and is left to the reader. One direction is seen by differentiating the expression $(\thf\tensor \thh)m^*(\thh)$. After doing the same thing for each module condition, one sees that they are collectively equivalent to the annihilation conditions in the five last rows.
\end{proof}

As promised, we can now prove the following:
\begin{proposition} \label{prop:subsume}
If $\thf$ and $\thg$ are morphisms and $\thh$ is a pseudo-derivation then all annihilation conditions are satisfied, and hence the module conditions are automatically satisfied.
\end{proposition}

\begin{proof}
By Proposition \ref{prop:annihilation conditions} \eqref{pd}, if $\thh$ is a pseudo-derivation then the annihilation conditions in the four first rows are satisfied. If $\thf$ and $\thg$ are morphisms, then the annihilation conditions in the three remaining rows follow from the conditions $\thf\thh = \thzr$ and $\thh\thg = \thzr$:
$$(\thf\tensor \thf)m^*(\thh) = m^*(\thf)m^*(\thh) = m^*(\thf\thh)  = \thzr,$$
$$(\thf\tensor \thh)m^*(\thg) = (\thf\tensor \thh)(\thg\tensor \thg) = \thf\thg\tensor \thh\thg = \thzr,$$
and so on. That the module conditions hold then follows from Proposition \ref{prop:annihilation conditions} \eqref{md}.
\end{proof}

\begin{proof}[Proof of Theorem \ref{thm:thick perturbation}]
By Proposition \ref{prop:subsume}, the module conditions are satisfied, so by Proposition \ref{prop:stability}, $\thfp$ and $\thgp$ are morphisms, $\thhp$ is a pseudo-derivation and $\thtp$ is a derivation. We need  to show that $t = \tht_1$ and $t' = \tht_1'$ are perturbations of $A$ and $B$, respectively, and that $\thDiag^\tht$ is a thick contraction. The $n^{th}$ level of the diagram $\thDiag^\tht$ is equal to the diagram
$$\Diag_n^{\tht_n}\colon \bigSDR{(A^{\tensor n})^{\tht_n}}{(B^{\tensor n})^{\tht_n'}}{\thf_n'}{\thg_n'}{\thh_n'}$$
obtained by perturbing the $n^{th}$ level $\Diag_n$ of the thick contraction $\thDiag$ using the perturbation $\tht_n$ of $A^{\tensor n}$. By the Basic Perturbation Lemma, $\tht_n'$ is a perturbation of $B^{\tensor n}$ and $\Diag_n^{\tht_n}$ is a contraction. In particular, $t$ and $t'$ are perturbations of $A$ and $B$, respectively. Furthermore, the relations $\thfp\thgp = \tho$, $\thfp\thhp = \thzr$, $\thhp\thhp = \thzr$ and $\thhp\thgp = \thzr$ hold because they do so levelwise. However, to verify that $\thDiag^{\tht}$ is a thick contraction, it is not enough to know that each individual level is a contraction, we will also need the fact that $\tht$ and $\thtp$ are derivations. Observe that
$$\partial(\thhp)_n = d_{(A^t)^{\tensor n}} \thh_n' + \thh_n' d_{(A^t)^{\tensor n}}.$$
Since $\tht$ is a derivation, the tensor product differential $d_{(A^t)^{\tensor n}}$ in $(A^t)^{\tensor n}$ coincides with the perturbed differential $d_{A^{\tensor n}} + \tht_n$ of $(A^{\tensor n})^{\tht_n}$. Since each $\Diag_n^{\tht_n}$ is a contraction, this implies that $\partial(\thhp) = \thgp\thfp - \tho$. Similarly, using that also $\thtp$ is a derivation one verifies that $\partial(\thfp) = \thzr$ and that $\partial(\thgp) = \thzr$. This finishes the proof.
\end{proof}

\begin{remark}
The reason for the name `module conditions' is the following:
Suppose that $A$ and $B$ are associative algebras and that $\thg_1\colon B\rightarrow A$ is a morphism of algebras. Then $A$ can be viewed as a left $B$-module via $\mu_A(\thg_1\tensor 1)\colon B\tensor A\rightarrow A$. Suppose moreover that $\mu_A\thh_2 = \thh_1\mu_A$. Then the module condition $\thh_2(\thg_1\tensor 1) = \thg_1\tensor \thh_1$ implies that $\thh_1$ is a morphism of $B$-modules (of degree 1). The other module conditions have similar interpretations.
\end{remark}

\section{Symmetric tensor trick} \label{section:symmetric tensor trick}
By Proposition \ref{prop:algebra contraction} any contraction $\Diag$ can be extended to a thick contraction $\thDiag$ where $\thf$ and $\thg$ are morphisms and $\thh$ is a $(\tho,\thg\thf)$-derivation. In this section we will symmetrize $\thh$ to obtain a \emph{symmetric} thick contraction $\thDiag^\Symgpd$ which extends $\Diag$. The symmetrized homotopy $\sh$ is no longer a $(\tho,\thg\thf)$-derivation, but we will show that it is a pseudo-derivation. Throughout this section we will assume that $\rationals\subseteq \kk$. This assumption is necessary, see Proposition \ref{prop:sym}.
\vskip8pt
Fix a contraction $\Diag$, and consider its extension to a thick contraction $\thDiag$ given by Proposition \ref{prop:algebra contraction}:
$$\thf_n = f^{\tensor n},\quad \thg_n = g^{\tensor n},\quad \thh_n = \sum_{i+1+j=n} 1^{\tensor i}\tensor h \tensor \pi^{\tensor j}.$$
Here $\pi = gf$. Evidently, the thick maps $\thf$ and $\thg$ are symmetric, but $\thh$ is not.
\begin{definition} \label{def:symmetrized homotopy}
The \defn{symmetrized tensor trick homotopy} $\sh\colon A\rightarrow A$ is the thick map defined by
$$\shn{n} = \frac{1}{n!}\sum_{\sigma\in \Symgp{n}} \thh_n^\sigma,$$
where $\thh_n^\sigma = \sigma^{-1} \thh_n \sigma$.
\end{definition}
The idea of symmetrizing the tensor trick homotopy appears in \cite{GLS1,HS,Huebschmann-Lie,Huebschmann-sh} and presumably in many other places, but the author is not aware of any written source where the formal properties of the symmetrized homotopy are worked out in detail. In particular, we believe that the discovery that $\sh$ is a pseudo-derivation is new, see Theorem \ref{thm:symmetric tensor trick} below.
\begin{proposition} \label{prop:sh-decomposition}
The symmetrized homotopy $\sh \colon A\rightarrow A$ can be decomposed as
$$\sh = \Qb \hh = \hh \Qb$$
where $\hh$ and $\Qb$ are the symmetric thick maps from $A$ to itself given by
$$\hhn{n} = \sum_{i+1+j = n} 1^{\tensor i}\tensor h \tensor 1^{\tensor j}$$
and
\begin{equation*}
\Qn{n} = \sum_{\epsilon\in \{0,1\}^n} \coeff{n}{|\epsilon|} \pi^{\epsilon_1} \tensor \ldots \tensor \pi^{\epsilon_n}.
\end{equation*}
Here, $\pi = gf$, $|\epsilon| = \epsilon_1 +\ldots + \epsilon_n$, and
$$\coeff{n}{k} = \frac{k!(n-1-k)!}{n!}$$
if $k<n$. We define $\coeff{n}{n} = 0$.
\end{proposition}

\begin{proof}
The $n^{th}$ component of the symmetrized homotopy is given by the formula
$$\shn{n} = \frac{1}{n!}\sum_{\sigma\in\Symgp{n}} \thh_n^\sigma,$$
where $\thh_n^\sigma =  \sigma^{-1} \thh_n \sigma$.
Every $\sigma\in \Symgp{n}$ determines a total order $<_\sigma$ of $\{1,\ldots,n\}$ by
$$i<_\sigma j \Longleftrightarrow \sigma(i)<\sigma(j).$$
We have that
$$\thh_n^\sigma = \sum_{j=1}^n \alpha_1\tensor \ldots \tensor \alpha_{j-1} \tensor h \tensor \alpha_{j+1} \tensor \ldots \tensor \alpha_n,$$
where, for $i\ne j$,
$$\alpha_i = \left\{ \begin{array}{ll} 1 & \mbox{if $i<_\sigma j$} \\ \pi & \mbox{if $j<_\sigma i$} \end{array} \right.$$
Therefore, the sum of all $\thh_n^\sigma$ is a linear combination of terms of the form
$$\pi^{\epsilon_1}\tensor \ldots \tensor \pi^{\epsilon_{j-1}} \tensor h \tensor \pi^{\epsilon_{j+1}} \tensor \ldots \tensor \pi^{\epsilon_n},$$
where $\epsilon_i\in\{0,1\}$. The coefficient of such a term is the number of total orders on the set $\{1,\ldots,n\}$ with the property that $j$ is the $j^{th}$ element and all elements of the set $\set{i}{\epsilon_i=0}$ precedes all elements of the set $\set{i}{\epsilon_i=1}$. The number of such orders is $k!(n-1-k)!$, where $k = |\set{i}{\epsilon_i=1}| = |\epsilon|$. Hence,
$$\shn{n} = \sum_{j=1}^n \sum_{\substack{\epsilon\in \{0,1\}^n \\ \epsilon_j = 0}} \coeff{n}{|\epsilon|} \pi^{\epsilon_1}\tensor \ldots \tensor \pi^{\epsilon_{j-1}} \tensor h \tensor \pi^{\epsilon_{j+1}} \tensor \ldots \tensor \pi^{\epsilon_n}.$$
Since $h\pi = \pi h = 0$, this may be written as $\sh = \hh \Qb  = \Qb \hh$, as claimed.
\end{proof}

\begin{remark}
Observe that since $h\pi = \pi h = 0$ it does not matter how $\coeff{n}{n}$ is defined, but we define it to be zero for definiteness.
\end{remark}

\begin{theorem} \label{thm:symmetric tensor trick}
The diagram
$$\thDiag^\Sigma\colon \SDR{A}{B}{\thf}{\thg}{\sh}$$
is a symmetric thick contraction which extends $\Diag$. Furthermore, $\thf$ and $\thg$ are morphisms and $\sh$ is a pseudo-derivation.
\end{theorem}

\begin{proof}
The relation $\partial(\sh) = \thg\thf-\tho$ follows from the relation $\partial(\thh) = \thg\thf - \tho$ because symmetrization is a morphism of chain complexes
$$\Hom_\Cat(A^{\tensor n},A^{\tensor n})\rightarrow \Hom_\Cat(A^{\tensor n},A^{\tensor n})^{\Symgp{n}}.$$
and because the thick map $\thg\thf-\tho$ is symmetric. The relation $\thf\thg=\tho$ is clear. By Proposition \ref{prop:sh-decomposition} we have $\sh = \Qb \hh = \hh \Qb$. Since $fh = 0$, $hg = 0$ and $hh = 0$, it follows that $\thf\hh = \thzr$, $\hh \thg = \thzr$ and $\hh\hh = \thzr$. Therefore,  $\thf\sh = \thf \hh \Qb = \thzr$, $\sh \thg = \Qb \hh \thg = \thzr$ and $\sh\sh = \Qb\hh \hh \Qb = \thzr$. We have thus verified that $\thDiag^\Sigma$ is a contraction.

The maps $\thf$ and $\thg$ are by definition the morphisms that extend $f$ and $g$. To prove that $\sh$ is a pseudo-derivation, it suffices by Proposition \ref{prop:annihilation conditions} to verify the annihilation conditions. To do this, use the decomposition $\sh = \Qb \hh = \hh \Qb$ and the fact that $\hh$ is a derivation that annihilates $\thf$, $\thg$ and $\hh$. For instance,
\begin{align*}
(\thf\tensor \sh) m^*(\sh)
& = (\thf\tensor \Qb \hh)m^*(\hh) m^*(\Qb) \\
& = (\thf\tensor \Qb\hh)(\hh \tensor \tho + \tho \tensor \hh) m^*(\Qb) \\
& = (-\thf\hh \tensor \Qb\hh + \thf\tensor \Qb\hh \hh)m^*(\Qb) \\
& = \thzr.
\end{align*}
The other annihilation conditions are verified in a similar manner.
\end{proof}

\begin{remark}
We have proved that $\sh$ is a pseudo-derivation and that $\thDiag^\Sigma$ satisfies the module conditions via Proposition \ref{prop:annihilation conditions} by verifying the annihilation conditions. The module conditions can also be verified directly. These verifications boil down to statements about the coefficients $\coeff{n}{k}$. For example, in proving that
$$(\thf\tensor \tho)m^*(\sh) = \thf\tensor \sh,$$
one comes across the statement that the equality
$$\sum_{j=0}^r \binom{r}{j} \coeff{n}{j+k} = \coeff{n-r}{k}$$
holds for all non-negative integers $r,k,n$ with $r+k<n$. Verifying directly that $\sh$ is a pseudo-derivation involves a similar but more complicated equality. It is quite interesting that these combinatorial equalities are consequences of Proposition \ref{prop:annihilation conditions}.
\end{remark}

\section{Operads and cooperads} \label{section:operads}
For the convenience of the reader we have included this section with standard definitions and facts about operads and cooperads. Most things in this section can be found in \cite{Fresse}, and the reader familiar with operads can safely skip this section, referring back for notation if necessary.

A \defn{symmetric sequence} is a collection $\Op = \{\Op(n)\}_{n\geq 0}$ where $\Op(n)$ is a chain complex with a right action of the symmetric group $\Symgp{n}$. The \defn{Schur functor} associated to a symmetric sequence $\Op$ is the functor $\Op[\dash]$ from the category of chain complexes to itself given on objects by
$$\Op[A] = \bigoplus_{n\geq 0} \Op(n)\tensor_{\Symgp{n}} A^{\tensor n},$$
and on morphisms $f\colon A\rightarrow B$ by
$$\Op[f] = \bigoplus_{n\geq 0} 1\tensor_\Symgp{n} f^{\tensor n}\colon \Op[A]\rightarrow \Op[B],$$
see \cite[\S 2.1.1]{Fresse}. There is a parallel story for \emph{non-symmetric} operads and cooperads; here one considers sequences $\Op = \{\Op(n)\}_{\geq 0}$ where $\Op(n)$ is just a chain complex without any $\Sigma_n$-action. In this case, one sets $\Op[A] =  \bigoplus_{n\geq 0} \Op(n)\tensor A^{\tensor n}$. All the constructions in this section have obvious non-symmetric analogs.

The \defn{tensor product} of two symmetric sequences $\Op$ and $\OpP$ is the symmetric sequence $\Op\tensor \OpP$ given by
$$(\Op \tensor \OpP)(n) = \bigoplus_{p+q=n} \Ind_{\Symgp{p}\times \Symgp{q}}^\Symgp{n} \Op(p)\tensor \OpP(q).$$
Here $\Ind_{\Symgp{p}\times \Symgp{q}}^\Symgp{n} \Op(p)\tensor \OpP(q)$ denotes the induced $\Symgp{n}$-representation. This tensor product has the property that there is an isomorphism of functors from $\Cat$ to itself
$$(\Op\tensor \OpP)[\dash] \cong \Op[\dash]\tensor \OpP[\dash],$$
and it makes the category of symmetric sequences into a symmetric monoidal dg-category, see \cite[\S 2.1]{Fresse}.

The \defn{composition product} of two symmetric sequences $\Op$ and $\OpP$ is the symmetric sequence
$$\Op\circ \OpP = \bigoplus_{n\geq 0} \Op(n) \tensor_\Symgp{n} \OpP^{\tensor n}.$$
The composition product has the property that there is an isomorphism of functors from $\Cat$ to itself
$$(\Op\circ \OpP)[\dash] \cong \Op[\OpP[\dash]],$$
and it makes the category of symmetric sequences into a monoidal category, see \cite[\S 2.2]{Fresse}. The unit object for the composition product is the symmetric sequence $\OpI$ with $\OpI(1) = \kk$ and $\OpI(n) = 0$ for $n\neq 1$. Concretely, elements of $(\Op\circ \OpP)(n)$ are linear combinations of formal composites
$$\nu\circ (\nu_1\tensor \ldots \tensor \nu_r) \sigma,$$
where
$$\nu\in \Op(r),\quad \nu_1\in \OpP(a_1),\ldots, \nu_r\in \OpP(a_r),\quad \sigma\in \Symgp{n},\quad a_1+\ldots+a_r = n,$$
These formal composites are subject to $\kk$-linearity in each variable and the equivariance conditions
\begin{align}
(\nu \tau)\circ (\nu_1\tensor \ldots\tensor \nu_r) & = \nu\circ (\nu_{\tau^{-1}(1)}\tensor \ldots \tensor \nu_{\tau^{-1}(r)}) \tau_{i_1,\ldots,i_r}, \\
\nu \circ (\nu_1 \tau_1 \tensor \ldots \tensor \nu_r \tau_r) & = \nu\circ (\nu_1\tensor \ldots \tensor \nu_r) \tau_1\sqcup \ldots \sqcup \tau_r.
\end{align}
Here $\tau_{i_1,\ldots,i_r}\in \Symgp{n}$ is the block permutation whose action is given by first dividing $\{1,2,\ldots,n\}$ into $r$ blocks of sizes $i_1,\ldots,i_r$ and then permuting the blocks according to $\tau\in \Symgp{r}$. If $\tau_j\in\Symgp{i_j}$ for $j=1,\ldots,r$, then $\tau_1\sqcup\ldots\sqcup \tau_r\in\Symgp{n}$ denotes the permutation which permutes the elements within the $j^{th}$ block according to $\tau_j$. The right action of $\Symgp{n}$ is given by formally multiplying to the right.

An \defn{operad} is a monoid in the monoidal category of symmetric sequences with the composition product, i.e., a symmetric sequence $\Op$ together with a multiplication $\gamma\colon \Op\circ \Op \rightarrow \Op$ and a unit $\eta\colon\OpI\rightarrow \Op$ satisfying associativity and unit axioms, see \cite[\S 3.1]{Fresse}. If $\Op$ is an operad then the associated Schur functor $\Op[\dash]$ becomes a monad \cite[Chapter VI]{MacLane}, and an \defn{algebra over $\Op$} is an algebra over the monad $\Op[\dash]$, i.e., an object $A$ together with a morphism $\gamma_A\colon \Op[A]\rightarrow A$ satisfying a unit and an associativity constraint, see \cite[\S 3.2]{Fresse}, \cite[p.\ 140]{MacLane}. An $\Op$-algebra structure on $A$ can equivalently be defined as a morphism of operads $\Op\rightarrow \End{A}$. The image of $\mu\in\Op(n)$ in $\End{A}(n) = \Hom_\kk(A^{\tensor n},A)$ is an operation $\mu_A\colon A^{\tensor n}\rightarrow A$.

A \defn{cooperad} $\Coop$ is a comonoid in the monoidal category of symmetric sequences with the composition product, i.e., a symmetric sequence $\Coop$ together with a coproduct $\Delta\colon \Coop\rightarrow \Coop\circ \Coop$ and a counit $\epsilon\colon \Coop\rightarrow \OpI$ satisfying coassociativity and counit axioms, see \cite[\S 1.2.17]{Fresse3},\cite[\S 1.7]{GJ} or \cite[\S 4.7]{LV}. If $\Coop$ is a cooperad, then the associated Schur functor $\Coop[\dash]$ becomes a comonad \cite[p.\ 139]{MacLane}, and a \defn{$\Coop$-coalgebra} is a coalgebra over this comonad, i.e., an object $A$ together with a morphism $\Delta_A\colon A\rightarrow \Coop[A]$ satisfying a coassociativity constraint, see \cite[\S 1.2.17]{Fresse3},\cite[\S 1.7]{GJ} or \cite[\S 4.7.4]{LV}.

Let $\Coop$ be a \emph{connected cooperad}, i.e., a cooperad satisfying $\Coop(0) = 0$ and $\Coop(1) = \kk$. By the description of elements of a composition product above and by the counit axiom for $\Coop$ we may write $\Delta(\nu)\in (\Coop\circ \Coop)(n)$ as
$$\Delta(\nu) = \nu\circ 1^{\tensor n} + 1 \circ \nu + \sum_{q=1}^p \nu^q \circ (\nu_1^q\tensor\ldots \tensor \nu_{r_q}^q)\sigma_q$$
for some $\nu^q\in \Coop(r_q)$, $\nu_i^q\in \Coop(a_i^q)$ and $\sigma_q\in \Sigma_n$, where $\sum_i a_i^q = n$, $2\leq r_q \leq n-1$, $1\leq a_i^q \leq n-1$ and $a_i^q>1$ for at least one $i$. We will sometimes use the shorter notation
$$\Delta(\nu) = \sum_{q=0}^{p+1} \nu_q'\circ \nu_q''$$
where $\nu_q'' = (\nu_1^q\tensor \ldots \tensor \nu_{r_q}^q)\sigma_q\in \Coop^{\tensor r_q}(n)$ for $0<q<p+1$ and where we let the $0$th and $(p+1)$st terms be $\nu\circ 1^{\tensor n}$ and $1 \circ \nu$, respectively.

Let $\Delta_{(1)}$ be the \emph{quadratic part} of $\Delta(\nu)$, by which we mean the sum of the terms in $\Delta(\nu)$ with $a_i^q>1$ for \emph{exactly one} $i$. This may be written in the form
$$\Delta_{(1)}(\nu) = \sum_{i=1}^u (\nu_i'\circ_{e_i} \nu_i'')\tau_i$$
for $\nu_i'\in \Coop(a_i')$, $\nu_i''\in\Coop(a_i'')$ and $\tau_i\in\Sigma_n$, where
$$\nu_i'\circ_{e_i} \nu_i'' = \nu_i'\circ (1^{\tensor e_i-1}\tensor \nu_i''\tensor 1^{\tensor a_i'-e_i}).$$

\section{Perturbation lemma for algebras over operads} \label{section:PL}
\begin{definition}
Let $\Op$ be an operad and let $A$, $B$ be $\Op$-algebras. We define a \defn{thick map of $\Op$-algebras} to be a symmetric thick map $\thf\colon A\rightarrow B$ such that the diagram
$$\xymatrix{\Op[A] \ar[d]^-{\gamma_A} \ar[r]^-{\Op[\thf]} & \Op[B] \ar[d]^-{\gamma_B} \\ A \ar[r]^-{\thf_1} & B }$$
commutes, where the upper horizontal map is given by
$$\Op[\thf] = \bigoplus_{n\geq 0} 1\tensor_\Symgp{n} \thf_n\colon \Op[A]\rightarrow \Op[B].$$
\end{definition}
In more elementary terms, a thick map of $\Op$-algebras $\thf\colon A\rightarrow B$ is a sequence
$$\thf  = \{\thf_n\colon A^{\tensor n}\rightarrow B^{\tensor n}\}_{n\geq 0}$$
of $\Symgp{n}$-equivariant maps of the same degree $|\thf|$ such that
$$\thf_1\mu_A = (-1)^{|\mu||\thf|}\mu_B \thf_n$$
for every $\mu\in \Op(n)$.

Thick maps of $\Op$-algebras simultaneously generalize morphisms and derivations (\cite[Definition 2.5]{GJ},\cite[\S 5.3.8]{LV}) of $\Op$-algebras. $\Op$-algebras together with thick maps of $\Op$-algebras form a dg-category that contains the ordinary category of of $\Op$-algebras as a subcategory.
\begin{proposition} \label{prop:TO}
Let $A$, $B$, $C$ be $\Op$-algebras.
\begin{itemize}
\item If $\thf,\thg\colon A\rightarrow B$ and $\thh\colon B\rightarrow C$ are thick maps of $\Op$-algebras, then so are
$\thh\circ \thf$, $\partial(\thf)$ and $a\thf + b\thg$, for $a,b\in \kk$.
In other words, $\Op$-algebras and thick maps of $\Op$-algebras form a dg-subcategory $T_\Op(\Cat)$ of the dg-category $T_\Symgpd(\Cat)$ of chain complexes and symmetric thick maps.

\item Morphisms of $\Op$-algebras $f\colon A\rightarrow B$ may be identified with thick maps of $\Op$-algebras $\thf\colon A\rightarrow B$ that satisfy $\thf_{p+q} = \thf_p\tensor \thf_q$ for all $p,q$.

\item Derivations of $\Op$-algebras $d\colon A\rightarrow A$ may be identified with thick maps of $\Op$-algebras $\thd\colon A\rightarrow A$ that satisfy $\thd_{p+q} = \thd_p\tensor\tho + \tho\tensor \thd_q$ for all $p,q$.
\end{itemize}
\end{proposition}

\begin{proof}
This is an exercise in the definitions.
\end{proof}

\begin{definition} \label{def:thick op-algebra contraction}
We define an \defn{$\Op$-algebra contraction} to be a contraction
$$\Diag \colon \SDR{A}{B}{\thf}{\thg}{\thh}$$
in the dg-category $T_\Op(\Cat)$, where $\thf$ and $\thg$ are morphisms and $\thh$ is a pseudo-derivation.
\end{definition}

\begin{proof}[Proof of Theorem \ref{thm:PL}]
By Theorem \ref{thm:thick perturbation}, $\Diag^\tht$ is a thick contraction, $\thfp$ and $\thgp$ are morphisms, $\thhp$ is a pseudo-derivation, $\thtp$ is a derivation, and $t = \tht_1$ and $t' = \tht_1'$ are perturbations of $A$ and $B$. We need to verify that the perturbed objects $A^t$ and $B^{t'}$ are $\Op$-algebras and that $\thfp$, $\thgp$ and $\thhp$ are thick maps of $\Op$-algebras between $A^t$ and $B^{t'}$.

Since $\tht$ and $\thh$ are thick maps of $\Op$-algebras from $A$ to itself, it follows that so are $\tho-\thh\tht$, the inverse $(\tho-\thh\tht)^{-1}$, and $\thS = \tht(\tho-\thh\tht)^{-1}$. Hence the perturbed maps $\thfp$, $\thgp$, $\thhp$ and $\thtp$, being built by composing and adding thick maps of $\Op$-algebras, are again thick maps of $\Op$-algebras, viewed as thick maps between $A$ and $B$.

Since $\tht\colon A\rightarrow A$ is a derivation and a thick map of $\Op$-algebras, $t$ is a derivation of $\Op$-algebras. Therefore $A^t$, which is just $A$ with perturbed differential $d_A + t$,  becomes an $\Op$-algebra with the same structure maps as $A$. Similarly, since $\thtp\colon B\rightarrow B$ is a derivation and a thick map of $\Op$-algebras, $B^{t'}$ is an $\Op$-algebra with the same structure maps as $B$.

Since the $\Op$-algebra structure maps for $A^t$ and $B^{t'}$ are the same as those for $A$ and $B$ respectively, the thick maps $\thfp$, $\thgp$, $\thhp$ and $\thtp$ are indeed thick maps of $\Op$-algebras between $A^t$ and $B^{t'}$.
\end{proof}

Invertibility of $\tho-\thh\tht$ can be ensured by having suitable filtrations on the objects.

\section{Tensor trick for algebras over operads} \label{section:TT}
\begin{proposition} \label{prop:Schur}
Let $\Op$ be a symmetric sequence. The associated Schur functor $\Op[\dash]\colon \Cat\rightarrow \Cat$ extends to a dg-functor $\Op_\bullet[\dash]\colon T_\Sigma(\Cat)\rightarrow T_\Sigma(\Cat)$. This extended Schur functor preserves morphisms and pseudo-derivations. If $\Op$ is an operad and if $\thf$ is any symmetric thick map then $\Op_\bullet[\thf]$ is a symmetric thick map of $\Op$-algebras.

Similarly, for a non-symmetric sequence $\Op$ there is an extension of the associated Schur functor to a dg-functor $\Op_\bullet[-]\colon T_\numbers(\Cat) \rightarrow T_\numbers(\Cat)$ which preserves morphisms and pseudo-derivations. If $\Op$ is a non-symmetric operad then $\Op_\bullet[\thf]$ is a thick map of $\Op$-algebras for any thick map $\thf$.
\end{proposition}

\begin{proof}
We will consider the symmetric case. The non-symmetric case is practically identical. The extension will be carried out in two steps. Firstly, note that the Schur functor $\Op[\dash]\colon \Cat \rightarrow \Cat$ extends to a dg-functor $T_\Sigma(\Cat)\rightarrow \Cat$. Indeed, if $\thf\colon A\rightarrow B$ is a symmetric thick map then let
$$\Op[\thf] = \bigoplus_{n\geq 0} 1\tensor_\Symgp{n} \thf_n\colon \Op[A]\rightarrow \Op[B].$$
It is straightforward to check that $\Op[\dash]$ is $\kk$-linear, that $\Op[\partial(\thf)] = \partial(\Op[\thf])$ and that $\Op[\thf \circ \thg] = \Op[\thf]\circ \Op[\thg]$.

Secondly, for a symmetric thick map $\thf \colon A\rightarrow B$, the $n^{th}$ level $\Op_n[\thf]$ of the thick map $\Op_\bullet[\thf]\colon \Op[A]\rightarrow \Op[B]$ is defined by requiring commutativity of the following diagram
$$\xymatrix{\Op[A]^{\tensor n} \ar[d]^-\cong \ar[rr]^-{\Op_n[\thf]} && \Op[B]^{\tensor n} \ar[d]^-\cong \\ \Op^{\tensor n}[A] \ar[rr]^-{\Op^{\tensor n}[\thf]} && \Op^{\tensor n}[B].}$$
Here, the lower horizontal map $\Op^{\tensor n}[\thf]$ is obtained by applying the dg-functor
$$\Op^{\tensor n}[\dash]\colon T_\Sigma(\Cat)\rightarrow \Cat$$
obtained in the first step to the symmetric thick map $\thf$. The vertical maps are given by the natural isomorphism $\Op[\dash]^{\tensor n} \cong \Op^{\tensor n}[\dash]$ of functors from $\Cat$ to itself. To be more explicit, observe that there is an isomorphism
$$\Op[A]^{\tensor n} \cong \bigoplus_{r_1,\ldots,r_n\geq 0} (\Op(r_1)\tensor \ldots \tensor \Op(r_n))\tensor_{\Symgp{r_1}\times \ldots \times \Symgp{r_n}} A^{\tensor (r_1 +\ldots + r_n)}.$$
On the summand indexed by $(r_1,\ldots,r_n)$, the map $\Op_n[\thf]$ acts as $\thf_{r_1+\ldots + r_n}$. The thick map $\Op_\bullet[\thf]$ is symmetric because $\Op \mapsto \Op[\dash]$ is a symmetric monoidal functor (see \cite[Proposition 2.1.5]{Fresse}). $\Op_\bullet[\dash]$ is a dg-functor because it is so at each level. Thus, we have obtained the required extension.

Suppose that $\thh\colon A\rightarrow A$ is a pseudo-derivation. We need to show that the thick map $\thH = \Op_\bullet[\thh]\colon \Op[A]\rightarrow \Op[A]$ is a pseudo-derivation. Indeed, for any $p,q$ the restriction of the map
$$(\thH_p\tensor \tho - \tho\tensor \thH_q)\thH_{p+q}\colon \Op[A]^{\tensor (p+q)}\rightarrow \Op[A]^{\tensor (p+q)}$$
to the summand indexed by $(r_1,\ldots,r_{p+q})$ acts on the right factor $A^{\tensor (r_1+\ldots + r_{p+q})}$ as
$$(\thh_i\tensor \tho - \tho\tensor \thh_j)\thh_{i+j}$$
where $i = r_1+\ldots+r_p$ and $j = r_{p+1}+\ldots+r_{p+q}$. Since $\thh $ is a pseudo-derivation, this is equal to $\thh_i\tensor \thh_j$. But this is exactly how the map $\thH_p\tensor \thH_q\colon \Op[A]^{\tensor (p+q)}\rightarrow \Op[A]^{\tensor (p+q)}$ restricted to the component indexed by $(r_1,\ldots,r_{p+q})$ acts on the right factor. Thus,
$$(\thH_p\tensor \tho - \tho\tensor \thH_q)\thH_{p+q} = \thH_p\tensor \thH_q.$$
By the same argument
$$-\thH_{p+q}(\thH_p\tensor \tho - \tho\tensor \thH_q) = \thH_p\tensor \thH_q,$$
and so $\thH$ is a pseudo-derivation. The proof that the dg-functor $\Op_\bullet[\dash]\colon T_\Sigma(\Cat)\rightarrow T_\Sigma(\Cat)$ takes morphisms to morphisms is similar.

Finally, suppose that $\Op$ is an operad and let $\thf\colon A\rightarrow B$ be any symmetric thick map. We need to show that $\thF  = \Op_\bullet[\thf]$ is a thick map of $\Op$-algebras. It is straightforward to check that the diagram
$$\xymatrix{\Op[\Op[A]] \ar[d]^-\cong \ar[rr]^-{\Op[\thF]} && \Op[\Op[B]] \ar[d]^-\cong \\ (\Op\circ \Op)[A] \ar[rr]^-{(\Op\circ \Op)[\thf]} && (\Op\circ \Op)[B]}$$
commutes. Since the $\Op$-algebra structure on $\Op[A]$ is given by the composite
$$\xymatrix{\Op[\Op[A]] \ar[r]^-\cong & (\Op\circ \Op)[A] \ar[r]^-{\gamma_\Op[A]} & \Op[A]},$$
see \cite[\S 3.2.13]{Fresse}, this implies that $\thF$ is a thick map of $\Op$-algebras.
\end{proof}

\begin{proof}[Proof of Theorem \ref{thm:tt}]
With $\thh$ as in Theorem \ref{thm:tt}, we get a contraction in the dg-category $T_\Sigma(\Cat)$
$$\SDR{A}{B}{\thf}{\thg}{\thh}$$
that extends the original contraction $\Diag$. Any dg-functor preserves contractions, so if we apply the extended Schur functor $\Op_\bullet[-]$ from Proposition \ref{prop:Schur} we get a contraction of $\Op$-algebras with the desired properties. The second part of Theorem \ref{thm:tt} follows from Proposition \ref{prop:algebra contraction} and Theorem \ref{thm:symmetric tensor trick}.
\end{proof}

We will now show the necessity of the assumption $\rationals\subseteq \kk$ in Theorem \ref{thm:symmetric tensor trick}.
\begin{proposition} \label{prop:sym}
If every contraction $\Diag$ can be extended to a symmetric thick contraction $\thDiag$ then necessarily $\rationals \subseteq \kk$.
\end{proposition}

\begin{proof}
For integers $n$ and $m$, let $\Disc{m}{n}$ denote the chain complex whose underlying graded $\kk$-module has one generator $x$ in degree $n$ and one generator $y$ in degree $n-1$, and where the differential is given by $d(x) = my$ and $d(y) = 0$. Defining $h\colon \Disc{1}{2}\rightarrow \Disc{1}{2}$ by $h(x) = 0$, $h(y) = x$, and $f = 0$, $g = 0$, we get a contraction
$$\Diag \colon \SDR{\Disc{1}{2}}{0}{f}{g}{h}.$$
If this had an extension to a symmetric thick contraction $\thDiag$, then we could apply $\Op[\dash]$ to this, for any symmetric sequence $\Op$. Consider the particular symmetric sequence $\OpS$ with $\OpS(0) = 0$ and $\OpS(n) = \kk$, the trivial representation of $\Symgp{n}$, for $n\geq 1$. The value at $A$ of the associated Schur functor is the (non-unital) symmetric algebra on $A$.
% $$\OpS[A] = \bigoplus_{n\geq 1} (A^{\tensor n})_\Symgp{n}.$$
Applying $\OpS[\dash]$ to the symmetric thick contraction $\thDiag$, we would get a contraction
$$\bigSDR{\OpS[\Disc{1}{2}]}{\OpS[0]}{F}{G}{H}.$$
But $\OpS[0] = 0$, so this can only happen if $\OpS[\Disc{1}{2}]$ is contractible. As a graded module, $\OpS[\Disc{1}{2}]$ has basis $x^{n+1}$, $x^ny$, where $|x^{n+1}|=2n+2$ and $|x^ny| = 2n+1$. The differential is given by $d(x^{n+1}) = (n+1)x^ny$ and $d(x^ny) = 0$, so there is a direct sum decomposition
$$\OpS[\Disc{1}{2}] \cong \bigoplus_{n\geq 0} \Disc{n+1}{2n+2}.$$
Therefore, $\OpS[\Disc{1}{2}]$ is contractible if and only if $\Disc{n+1}{2n+2}$ is contractible for all $n\geq 0$. But $\Disc{n+1}{2n+2}$ is contractible if and only if $n+1$ is invertible in $\kk$. Hence, $\OpS[\Disc{1}{2}]$ is contractible if and only if $\rationals \subseteq \kk$.
\end{proof}

\section{Perturbation lemma and tensor trick for coalgebras over cooperads} \label{section:cooperads}
In this section we will dualize the results of the previous sections. The proofs are virtually the same and will therefore be omitted.

\begin{definition}
Let $\Coop$ be a cooperad and let $A$ and $B$ be $\Coop$-coalgebras. We define a  \defn{thick map of $\Coop$-coalgebras} to be a symmetric thick map $\thf\colon A\rightarrow B$ such that the diagram
$$\xymatrix{A \ar[d]^-{\Delta_A} \ar[r]^-{f_1} & B \ar[d]^-{\Delta_B} \\ \Coop[A] \ar[r]^-{\Coop[\thf]} & \Coop[B]}$$
commutes.
\end{definition}

\begin{proposition} \label{prop:TC}
Let $A$,$B$,$C$ be $\Coop$-coalgebras.
\begin{itemize}
\item If $\thf,\thg\colon A\rightarrow B$ and $\thh\colon B\rightarrow C$ are thick maps of $\Coop$-coalgebras, then so are
$\thh\circ \thf$, $\partial(\thf)$ and $a\thf + b\thg$, for $a,b\in \kk$. In other words, $\Coop$-coalgebras and thick maps of $\Coop$-coalgebras form a dg-subcategory $T_\Coop(\Cat)$ of the dg-category $T_\Sigma(\Cat)$ of chain complexes and symmetric thick maps.

\item Morphisms of $\Coop$-coalgberas $\thf\colon A\rightarrow B$ may be identified with thick maps of $\Coop$-coalgebras $\thf\colon A\rightarrow B$ that satisfy $\thf_{p+q} = \thf_p\tensor \thf_q$ for all $p,q$.

\item Coderivations of $\Coop$-coalgebras $d\colon A\rightarrow A$ may be identified with thick maps of $\Coop$-coalgebras $\thd\colon A\rightarrow A$ that satisfy $\thd_{p+q} = \thd_p\tensor\tho + \tho\tensor \thd_q$ for all $p,q$.
\end{itemize}
\end{proposition}

\begin{definition} \label{def:thick coop-coalgebra contraction}
We define a \defn{contraction of $\Coop$-coalgebras} to be a contraction
$$\Diag \colon\SDR{A}{B}{\thf}{\thg}{\thh}$$
in the dg-category $T_\Coop(\Cat)$, where $\thf$ and $\thg$ are morphisms and $\thh$ is a pseudo-derivation.
\end{definition}

\begin{theorem}[$\Coop$-coalgebra Perturbation Lemma] \label{thm:TCPL}
Let $\Diag$ be a contraction of $\Coop$-coalgebras. If $\tht$ is a perturbation of $A$ then, provided the series $\tht+\tht\thh\tht +\ldots$ converges, the recursive formulas
\begin{align*}
\thfp & = \thf + \thfp\tht\thh, & \thgp & = \thg + \thh\tht\thgp, \\
\thhp & = \thh + \thhp\tht\thh, & \thtp & = \thf\tht\thgp,
\end{align*}
define a perturbation $\thtp$ of $B$ and a contraction of $\Coop$-coalgebras
$$\Diag^\tht \colon \bigSDR{(A,d_A+\tht_1)}{(B,d_B+\tht_1')}{\thfp}{\thgp}{\thhp}$$
In particular, $\thfp$ and $\thgp$ are morphisms, $\thtp$ is a coderivation and $\thhp$ is a pseudo-derivation.
\end{theorem}

\begin{theorem}[$\Coop$-coalgebra Tensor Trick] \label{thm:cooperad tensor trick}
Consider a contraction of chain complexes
$$\Diag\colon\SDR{A}{B}{f}{g}{h}$$
For any choice of symmetric pseudo-derivation $\thh$ such that $\thh_1 = h$ and $\partial(\thh)= \thg\thf-\tho$, $\thh\thh = \thzr$, $\thf\thh = \thzr$, $\thh\thg=\thzr$, where $\thf_n =f^{\tensor n}$ and $\thg_n = g^{\tensor n}$, there is an induced contraction of $\Coop$-coalgebras
$$\Coop_\bullet[\Diag] \colon \bigSDR{\Coop[A]}{\Coop[B]}{\Coop_\bullet[\thf]}{\Coop_\bullet[\thg]}{\Coop_\bullet[\thh]}.$$
If $\Coop$ is a non-symmetric cooperad, then one may drop the condition that $\thh$ be symmetric. There is always a non-symmetric pseudo-derivation $\thh$ with the requisite properties, namely
$$\thh_n = \sum_{p+1+q = n} 1^{\tensor p}\tensor h\tensor (gf)^{\tensor q}.$$
If $\rationals\subseteq \kk$ then, with $\thh_n$ as above,
$$\thh_n^\Sigma = \frac{1}{n!}\sum_{\sigma\in \Symgp{n}} \sigma^{-1} \thh_n \sigma,$$
defines a symmetric pseudo-derivation $\thh^\Sigma$ with the requisite properties.
\end{theorem}

\section{Thick maps between `cofree' $\Coop$-coalgebras} \label{section:maps}
\begin{proposition} \label{prop:bridge}
Let $\Coop$ be a connected cooperad. A symmetric thick map of $\Coop$-coalgebras $\thF\colon \Coop[A]\rightarrow \Coop[B]$ determines maps
$$\thF^\nu\colon A^{\tensor n}\rightarrow B^{\tensor m}$$
for $\nu\in\Coop^{\tensor m}(n)$, $m,n\geq 1$, such that the $\kk$-linear structure, differentials and symmetric group actions are respected in the sense that
\begin{align*}
\thF^{a\nu + b\nu'} & = a\thF^\nu + b\thF^{\nu'}, \\
(a\thF + b\thG)^\nu & = a\thF^\nu + b\thG^\nu, \\
\partial(\thF^\nu) & = \partial(\thF)^\nu +(-1)^{|\thF|} \thF^{d(\nu)}, \\
\thF^{\tau \nu \sigma} & = \tau \thF^\nu \sigma,
\end{align*}
for any symmetric thick maps $\thF,\thG\colon \Coop[A]\rightarrow \Coop[B]$, and any $\nu,\nu'\in\Coop^{\tensor m}(n)$, $a,b\in\kk$, $\tau\in\Sigma_m$, $\sigma\in\Sigma_n$. Composition of thick maps is respected in the sense that for any $\nu\in\Coop(n)$
\begin{equation*}
(\thF \thG)^\nu = \sum_q (-1)^{|\thG||\nu_q'|} \thF^{\nu_q'} \thG^{\nu_q''}.
\end{equation*}
where $\Delta(\nu) = \sum_q \nu_q' \circ \nu_q''\in (\Coop\circ\Coop)(n)$ for $\nu_q'\in\Coop(r_q)$ and $\nu_q''\in\Coop^{\tensor r_q}(n)$. Furthermore,
\begin{enumerate}
\item $\thF_1$ is determined by the collection of maps $\thF^\nu$, for $\nu\in \Coop^{\tensor m}(n)$.
\item If $\thF$ is a morphism of $\Coop$-coalgebras then $\thF^{\nu_1\tensor \ldots\tensor\nu_m} = \thF^{\nu_1}\tensor \ldots \tensor \thF^{\nu_m}$ for any $\nu_1,\ldots,\nu_m\in\Coop$. In particular, $\thF$ is determined by the collection of maps $\thF^\nu$ for $\nu\in\Coop(n)$.

\item If $\tht\colon \Coop[A]\rightarrow \Coop[A]$ is a weight decreasing coderivation then $\tht^{\nu_1\tensor\ldots\tensor\nu_m} = 0$ unless $\nu_i$ has arity $ > 1$ for exactly one $i$ and $\tht^{1^{\tensor i}\tensor \nu\tensor 1^{\tensor j}} = 1^{\tensor i}\tensor \tht^\nu \tensor 1^{\tensor j}$. In particular, $\tht$ is determined by the collection of maps $\tht^\nu$ for $\nu\in\Coop(n)$. \label{item:coderivation}

\item If $\thF$ is induced by a symmetric thick map $\thf\colon A\rightarrow B$ then $\thF^\nu = 0$ unless $m=n$, and for $1\in \Coop^{\tensor m}(m)$ we have $\thF^1 = \thf_m$.
\end{enumerate}
\end{proposition}

\begin{proof}
For $\nu\in\Coop^{\tensor m}(n)$, let $\iota_\nu\colon A^{\tensor n} \rightarrow \Coop^{\tensor m}[A]$ be the map $a\mapsto \nu\tensor a$. Let $\epsilon\colon \Coop[A]\rightarrow A$ denote the map induced by the counit of $\Coop$. Given a symmetric thick map of $\Coop$-coalgebras $\thF\colon \Coop[A]\rightarrow \Coop[B]$, we let $\thF^\nu$ be the composite
$$\xymatrix{A^{\tensor n} \ar[r]^-{\iota_\nu} & \Coop^{\tensor m}[A] \cong \Coop[A]^{\tensor m} \ar[r]^-{\thF_m} & \Coop[B]^{\tensor m} \ar[r]^-{\epsilon^{\tensor m}} & B^{\tensor m}}$$
$$\thF^\nu = \epsilon^{\tensor m}\thF_m\iota_{\nu}.$$
From this description, it is immediate that the $\kk$-linear structures, symmetric group actions and differentials are respected. We get a commutative diagram
$$\xymatrix{\Coop[A]\ar[r]^-{\thF_1} \ar[d]^-\Delta & \Coop[B] \ar[d]^-\Delta \ar@{=}[dr] \\ \Coop[\Coop[A]] \ar[r]_-{\Coop[\thF]} & \Coop[\Coop[B]] \ar[r]_-{\Coop[\epsilon]} & \Coop[B]}$$
where the square commutes by definition of thick maps of $\Coop$-coalgebras and the triangle commutes because of the counit axiom for cooperads. This shows that $\thF_1$ is determined by the collection of maps $\epsilon^{\tensor m}\thF_m$, and hence also by the maps $\thF^\nu$ for $\nu\in\Coop^{\tensor m}(n)$. The above diagram shows moreover that
\begin{align*}
\thF_1 \iota_\nu(a)  & = \Coop[\mathbf{\epsilon}\thF]\Delta (\nu\tensor a) = \sum_q (1\tensor \epsilon^{\tensor \alpha_q} \thF_{\alpha_q}) (\nu_q'\tensor \nu_q''\tensor a) \\
& = \sum_q (-1)^{|\thF||\nu_q'|} \nu_q'\tensor \epsilon^{\tensor \alpha_q} \thF_{\alpha_q} (\nu_q''\tensor a),
\end{align*}
for any $\nu\in\Coop(n)$ and any $a\in A^{\tensor n}$, where $\Delta(\nu) = \sum_q \nu_q' \circ \nu_q''$ for $\nu_q'\in\Coop(r_q)$ and $\nu_q''\in\Coop^{\tensor r_q}(n)$. In other words, $\thF_1 \iota_\nu = \sum_q (-1)^{|\nu_q'||\thF|} \iota_{\nu_q'} \thF^{\nu_q''}$. Thus,
$$(\thF \thG)^\nu = \epsilon\thF_1 \thG_1 \iota_\nu = \sum_q (-1)^{|\nu_q'||\thG|}\epsilon\thF \iota_{\nu_q'} \thG^{\nu_q''} = \sum_q (-1)^{|\thG||\nu_q'|} \thF^{\nu_q'} \thG^{\nu_q''}.$$
The remaining properties are straightforward to check.
\end{proof}

\section{Transferring $\Omega(\Coop)$-algebra structures} \label{section:transfer}
Let $\Coop$ be a connected cooperad and let $\Omega(\Coop)$ denote the cobar construction \cite{Fresse-cobar,GJ}.
The following two propositions are well-known, see \cite{Fresse-cobar,GJ}. They can also be proved easily using Proposition \ref{prop:bridge}.

\begin{proposition} \label{prop:infty-algebra}
An $\Omega(\Coop)$-algebra structure on a chain complex $A$ is described by any of the following
\begin{itemize}
\item A weight decreasing coderivation $t\colon \Coop[A]\rightarrow \Coop[A]$ of degree $-1$ satisfying $\partial(t) + t^2 =0$.

\item Maps $t^\nu\colon A^{\tensor n}\rightarrow A$ of degree $|\nu|-1$ for all $\nu\in\Coop(n)$, $n\geq 1$, satisfying
\begin{equation*}
\partial(t^\nu) + t^{d(\nu)} + \sum_{i=0}^u (-1)^{|\nu_i'|} (t^{\nu_i'}\circ_{e_i} t^{\nu_i''})\tau_i = 0,
\end{equation*}
and $t^{a\nu + b\nu'} = at^\nu + bt^{\nu'}$, $t^{\nu\sigma} = t^\nu \sigma$ for $\nu,\nu'\in\Coop(n)$, $a,b\in\kk$ and $\sigma\in\Sigma_n$.
\end{itemize}
\end{proposition}

\begin{proposition}
Let $t$ and $t'$ be $\Omega(\Coop)$-algebra structures on $A$ and $B$, respectively. A lax morphism of $\Omega(\Coop)$-algebras $(A,t)\rightarrow (B,t')$ is described by any of the following
\begin{itemize}
\item A morphism of $\Coop$-coalgebras $f\colon \Coop[A]\rightarrow \Coop[B]$ satisfying
$$(d_{\Coop[B]} + t')f = f(d_{\Coop[A]} + t).$$
\item Maps $f^\nu\colon A^{\tensor n}\rightarrow B$, of degree $|\nu|$ for all $\nu\in\Coop(n)$, $n\geq 1$, satisfying (with $f := f^1$)
\begin{align*}
\partial(f^\nu) + (t')^\nu f^{\tensor n} & + \sum_{q=1}^p (t')^{\nu^q}(f^{\nu_1^q}\tensor \ldots \tensor f^{\nu_{r_q}^q}) \sigma_q \\
& = f^{d(\nu)} + f t^\nu + \sum_{i=1}^u (-1)^{|\nu_i'|}(f^{\nu_i'}\circ_{e_i} t^{\nu_i''})\tau_i,
\end{align*}
and $f^{a\nu + b\nu'} = af^\nu + bf^{\nu'}$, $f^{\nu\sigma} = f^\nu \sigma$ for $\nu,\nu'\in\Coop(n)$, $a,b\in\kk$ and $\sigma\in\Sigma_n$.
\end{itemize}
\end{proposition}

\begin{proof}[Proof of Theorem \ref{thm:implicit} and Theorem \ref{thm:transfer}]
By the $\Coop$-coalgebra Tensor Trick, Theorem \ref{thm:cooperad tensor trick}, a suitable choice of pseudo-derivation $\thh$ gives rise to a contraction of $\Coop$-coalgebras
$$\bigSDR{\Coop[A]}{\Coop[B]}{\Coop_\bullet[\thf]}{\Coop_\bullet[\thg]}{\Coop_\bullet[\thh]}.$$
For ease of notation, let $\thF =\Coop_\bullet[\thf]$, $\thG = \Coop_\bullet[\thg]$, $\thH = \Coop_\bullet[\thh]$. The $\Omega(\Coop)$-algebra structure on $A$ is encoded in a weight decreasing coderivation perturbation $t\colon \Coop[A]\rightarrow \Coop[A]$. That $t$ is a coderivation perturbation implies that the thick map $\tht$ with
$$\tht_n = \sum_{p+1+q=n} 1^{\tensor p}\tensor t\tensor 1^{\tensor q}$$
is a thick map of $\Coop$-coalgebras that satisfies $\partial(\tht)+\tht^2 = \thzr$. Now, we can apply the $\Coop$-coalgebra Perturbation Lemma, Theorem \ref{thm:TCPL}, to obtain a new contraction of $\Coop$-coalgebras
$$\bigSDR{(\Coop[A],d_{\Coop[A]}+t)}{(\Coop[B],d_{\Coop[B]}+t')}{\thFp}{\thGp}{\thHp}$$
determined by the recursive formulas
\begin{align*}
\thFp & = \thF + \thFp\tht\thH, & \thGp & = \thG + \thH\tht\thGp, \\
\thHp & = \thH + \thHp\tht\thH, & \thtp & = \thF\tht\thGp.
\end{align*}
This proves Theorem \ref{thm:implicit}.

To prove Theorem \ref{thm:transfer}, we need to expand the above formulas. Let us write $f^\nu = (\thFp)^\nu$, $g^\nu = (\thGp)^\nu$, $t^\nu = \tht^\nu$, $(t')^\nu = (\thtp)^\nu$, $h^\nu = (\thHp)^\nu$. Observe that $f^1 = f$, $g^1 = g$, $h^1 = h$ and $t^1 = 0$ since $t$ decreases weight. By Proposition \ref{prop:bridge}, we have that for any $\nu\in \Coop(n)$ where $n>1$
\begin{align*}
(\thGp)^\nu & = (\thG + \thH\tht\thGp)^\nu \\
& = \thG^\nu + h (\tht\thGp)^\nu \\
& = \sum_{q=0}^{p+1} h t^{\nu^q} (\thGp)^{(\nu_1^q\tensor \ldots\tensor \nu_{r_q}^q) \sigma_q} \\
& = h t^\nu g^{\tensor n} + \sum_{q=1}^p h t^{\nu^q} (g^{\nu_1^q}\tensor \ldots \tensor g^{\nu_{r_q}^q})\sigma_q.
\end{align*}
The recursive formula for $(t')^\nu$ is derived in the same way. Similarly,
\begin{align*}
(\thFp)^\nu & = (\thF + \thF' \tht \thH)^\nu \\
& = \thF^\nu + (-1)^{|\nu|}(\thF'\tht)^\nu \thh_n \\
& = (-1)^{|\nu|} \sum_{q=0}^{p+1} (-1)^{|\nu^q|} f^{\nu^q} t^{\nu_1^q\tensor \ldots \tensor \nu_{r_q}^q} \thh_n \\
& = (-1)^{|\nu|} f t^\nu \HH{n} + \sum_{i=1}^u (-1)^{|\nu_i''|} (f^{\nu_i'}\circ_{e_i} t^{\nu_i''})\tau_i \HH{n},
\end{align*}
where we have used Proposition \ref{prop:bridge}\eqref{item:coderivation} in the last step. The recursive formula for $h^\nu$ is derived in the same way.
\end{proof}

\section{Example: $A_\infty$-algebras} \label{section:tree}
Let us illustrate how the formulas of Theorem \ref{thm:transfer} work in the case of $A_\infty$-algebras. $A_\infty$-algebras are exactly $\Omega(\Asa)$-algebras and $A_\infty$-morphisms are exactly lax morphisms of $\Omega(\Asa)$-algebras, where $\Asa = (\Lambda \As)^\vee$ is the Koszul dual cooperad of the associative operad $\As$. For $n\geq 1$, $\Asa(n)$ is the free right $\kk\Symgp{n}$-module on one generator $\mu_n$ of degree $n-1$. Write $\mu_1 = 1$. The differential is zero and the coproduct is given by
$$\Delta(\mu_n) = \sum_{i_1+\ldots+i_r = n} (-1)^\epsilon \mu_r \circ (\mu_{i_1}\tensor \ldots \tensor \mu_{i_r}),$$
where the sign is given by
$$\epsilon = \sum_{j<k} i_j(i_k-1).$$
The quadratic part of the coproduct is thus given by
$$\Delta_{(1)}(\mu_n) = \sum_{r+s+t=n} (-1)^{r(s-1)}\mu_u \circ (1^{\tensor r}\tensor \mu_s \tensor 1^{\tensor t}),$$
where $u = r+1+t$. Thus, in view of Proposition \ref{prop:infty-algebra} an $\Omega(\Asa)$-algebra structure $t$ on a chain complex $A$ is the same thing as a sequence of maps $m_n := t^{\mu_n}\colon A^{\tensor n}\rightarrow A$ of degree $n-2$ such that for all $n\geq 2$
$$\delta(m_n) = \sum_{r+s+t=n} (-1)^{rs + r + u}m_u\circ (1^{\tensor r}\tensor m_s\tensor 1^{\tensor t}).$$
Noting that $r+u$ has the same parity as $t+1$ in the above, this recovers the usual definition of an $A_\infty$-algebra with the same sign convention as in \cite[D\'efinition 1.2.1.1]{LH}. In the transfer theorem, writing $f_n := f^{\mu_n}$ etc., we see that
\begin{align*}
m_n' & = \sum_{\substack{i_1+\ldots + i_r = n \\ r>1}} (-1)^\epsilon f m_r (g_{i_1}\tensor \ldots \tensor g_{i_r}), \\
g_n & = \sum_{\substack{i_1+\ldots + i_r = n \\ r>1}} (-1)^\epsilon h m_r (g_{i_1}\tensor \ldots \tensor g_{i_r}), \\
f_n & = \sum_{\substack{p+u+q = n \\ r = p+1+q}} (-1)^{(p+1)(u+1)} f_r\big( 1^{\tensor p} \tensor m_u \tensor 1^{\tensor q} \big) \HH{n}, \\
h_n & = \sum_{\substack{p+u+q = n \\ r = p+1+q}} (-1)^{(p+1)(u+1)} h_r\big( 1^{\tensor p} \tensor m_u \tensor 1^{\tensor q} \big) \HH{n}.
\end{align*}
One choice of pseudo-derivation $\HH{n}$ is given by
$$\HH{n} = \sum_{i+1+j = n} 1^{\tensor i}\tensor h\tensor (gf)^{\tensor j},$$
but other choices are possible. The sign $(-1)^\epsilon$ in the formulas for $m_n'$ and $g_n$ is the same as in the description of the coproduct above.

If one unwinds these recursive formulas then one obtains `tree formulas'. To make this idea precise, let us see how this works for $g_n$ for low values of $n$. We have $g_2 = hm_2(g_1\tensor g_1) = hm_2(g\tensor g)$, and this may be represented pictorially as follows:
$$g_2 = \ttree$$
Next,
\begin{align*}
g_3 & = hm_3(g_1\tensor g_1\tensor g_1) + hm_2(g_2\tensor g_1) -hm_2(g_1\tensor g_2) \\
& = hm_3(g\tensor g\tensor g) + hm_2(hm_2(g\tensor g)\tensor g) - hm_2(g\tensor hm_2(g\tensor g)).
\end{align*}
This may be represented by the picture
$$g_3 = \rtree - \ltree + \mtree$$
In general, we have that $g_n$ is the alternating sum over all trees $T$ with $n$ leaves, where the leaves are decorated by $g$, the vertices by $m_r$, where $r$ is the number of incoming edges of the vertex at hand, and the root by $h$. The sign attached to a tree $T$ is determined by the parity of the number of pairs $(\ell,v)$ where $\ell$ is a leaf and $v$ is a vertex with an even number of incoming edges such that $\ell$ is to the left of $v$ in $T$. The formula for $m_n'$ is the same except that the root is decorated by $f$ instead of $h$. These are exactly the formulas written down by Kontsevich-Soibelman \cite[\S 6.4]{KS}, based on \cite{Merkulov}.

\end{document}